\documentclass[12pt,a4paper,twoside]{amsart}

\usepackage[english]{babel}
\usepackage[latin1]{inputenc}
\usepackage[T1]{fontenc}

\usepackage{enumerate}

\usepackage{amssymb}

\allowdisplaybreaks

\usepackage[DIV=14,BCOR=2cm]{typearea}

\usepackage{hyperref}

\newcommand{\IR}{\mathbb{R}}
\newcommand{\IN}{\mathbb{N}}
\newcommand{\IZ}{\mathbb{Z}}
\newcommand{\IC}{\mathbb{C}}
\newcommand{\IQ}{\mathbb{Q}}

\newcommand{\IH}{\mathbb{H}}

\newcommand{\frake}{\mathfrak{e}}

\newcommand{\frakv}{\mathfrak{v}}
\newcommand{\frakw}{\mathfrak{w}}
\newcommand{\calA}{\mathcal{A}}

\newcommand{\calE}{\mathcal{E}}
\newcommand{\calF}{\mathcal{F}}
\newcommand{\calG}{\mathcal{G}}

\newcommand{\calK}{\mathcal{K}}
\newcommand{\calL}{\mathcal{L}}

\newcommand{\calN}{\mathcal{N}}

\newcommand{\calR}{\mathcal{R}}
\newcommand{\calS}{\mathcal{S}}

\newcommand{\calW}{\mathcal{W}}

\renewcommand{\Re}{\operatorname{Re}}

\renewcommand{\Im}{\operatorname{Im}}

\newcommand{\sgn}{\operatorname{sgn}}

\newcommand{\Mp}{\operatorname{Mp}}

\newcommand{\res}{\operatorname{res}}
\newcommand{\Gr}{\operatorname{Gr}}
\newcommand{\SL}{\operatorname{SL}}

\newcommand{\bs}{\ensuremath{\backslash}}

\newcommand{\vol}{\operatorname{vol}}
\newcommand{\Eis}{\operatorname{Eis}}
\newcommand{\M}{\operatorname{M}}
\renewcommand{\S}{\operatorname{S}}
\newcommand{\Iso}{\operatorname{Iso}}
\newcommand{\Inv}{\operatorname{Inv}}

\theoremstyle{definition}
\newtheorem{defn}{Definition}[section]

\theoremstyle{plain}
\newtheorem{thm}[defn]{Theorem}
\newtheorem{lem}[defn]{Lemma}
\newtheorem{cor}[defn]{Corollary}
\newtheorem{prop}[defn]{Proposition}

\theoremstyle{remark}
\newtheorem{rem}[defn]{Remark}

\begin{document}
	
\title[Orthogonal Eisenstein Series of Singular Weight]{Orthogonal Eisenstein Series at Harmonic Points and Modular Forms of Singular Weight}
\author{Paul Kiefer}
\thanks{The author was partially supported by the LOEWE research unit USAG}
\thanks{Partially funded by the Deutsche Forschungsgemeinschaft (DFG, German Research Foundation) TRR 326 \textit{Geometry and Arithmetic of Uniformized Structures}, project number 444845124.}

\begin{abstract}
\noindent
We investigate the behaviour of orthogonal non-holomorphic Eisenstein series at their harmonic points by using theta lifts. In the case of singular weight, we show that the orthogonal non-holomorphic Eisenstein series that can be written as a theta lift have a simple pole at $s = 1$ whose residues yield holomorphic orthogonal modular forms that are Eisenstein series on the boundary and we give a sufficient condition on the surjectivity of this construction.
\end{abstract}

\maketitle

\tableofcontents

\section{Introduction}

Let $L$ be an even lattice in a rational quadratic space $V$ of signature $(2, l)$. There is an index $2$ subgroup of the corresponding orthogonal group $O^+(2, l) \subseteq O(2, l)$ acting on the orthogonal upper half-plane $\IH_l$. Similar to the case of elliptic modular forms, it is possible to define orthogonal modular forms, which can be seen to be global sections of a hermitian line bundle. For $l = 1$ we obtain the classical case of elliptic modular forms and for $l = 2$ we obtain Hilbert modular forms for real quadratic number fields.

An important problem is the construction of such modular forms, in particular for low weight. It turns out that there is a minimal weight for non-zero holomorphic modular forms, which is given by $\kappa = \frac{l}{2} - 1$ for $l > 2$, see \cite{Bundschuh}. The weight $\kappa = \frac{l}{2} - 1$ is called the singular weight, and modular forms of singular weight have many vanishing Fourier coefficients. In particular, there are no cusp forms of singular weight (for an analogous theory in the Siegel case see \cite{FreitagSiegel}, where it is shown that every holomorphic modular form of singular weight is a linear combination of theta functions). In contrast to the symplectic case, there are no holomorphic theta series for the orthogonal group (except for some low dimensional examples where we have an exceptional isomorphism from the orthogonal group to the symplectic group). Moreover, the usual constructions of holomorphic modular forms do not work in low weight. For example Eisenstein series of low weight do not converge. On the other hand, using the celebrated multiplicative Borcherds lift of \cite[Theorem 13.3]{Borcherds}, examples of holomorphic modular forms of singular weight can be constructed. For results in this direction see \cite{DittmannSingularWeight}, \cite{SchwagenscheidtSingularWeight}, \cite{ScheithauerSingularWeight}. Another method is to use the additive Borcherds lift of \cite[Theorem 14.3]{Borcherds} using holomorphic modular forms of weight $0$, i.e. invariant vectors, as input functions.

Our aim is to investigate Eisenstein series $\calE_{\kappa, \lambda}(Z)$ of low weight, in particular of singular weight $\kappa = \frac{l}{2} - 1$ by considering the non-holomorphic Eisenstein series $\calE_{\kappa, \lambda}(Z, s)$, get a meromorphic continuation to all $s \in \IC$ and hope that they yield holomorphic modular forms for special values of $s$. This will be done using the results of \cite{Kiefer}.

We will give more details now. As above, let $L$ be an even lattice of signature $(2, l)$ and $L'$ the corresponding dual lattice. Throughout we will assume that $L$ has Witt rank $2$ and that $l > 2$. For a primitive isotropic vector $z \in L$ and $z' \in L'$ with $(z, z') = 1$ consider $K = L \cap z^\perp \cap z'^\perp$. Let $\IH_l = K \otimes \IR + i C$, where $C$ is a fixed connected component of
$$\{Y \in K \otimes \IR \mid q(Y) > 0 \}.$$
Then $\IH_l$ is a hermitian symmetric domain and there is an index $2$ subgroup $O^+(L \otimes \IR) \subseteq O(L \otimes \IR)$ acting on $\IH_l$. Let $\Gamma(L)$ be the kernel of the natural map $O^+(L) \to O(L' / L)$. By the theory of Baily-Borel (see \cite{BailyBorel}, \cite{BorelJi}) we can add $0$-dimensional boundary components (points) and $1$-dimensional boundary components (isomorphic to $\IH$) to obtain $\IH_l^*$, such that the quotient $\Gamma(L) \bs \IH_l^*$ is a compact complex analytic space, which can be shown to be projective and hence, by the Chow's theorem, algebraic. For an holomorphic orthogonal modular form $F : \IH_l \to \IC$ of weight $\kappa \in \IZ$ with respect to $\Gamma(L)$ one can now define its restriction $F \vert_I : \IH \to \IC$ to a $1$-dimensional boundary component $I$. One easily sees that $F \vert_I$ is a holomorphic modular form of weight $\kappa$ for an appropriate arithmetic subgroup of $\SL_2(\IQ)$. In particular, it makes sense to talk about the space $\M_\kappa^{\partial \Eis}(\Gamma(L))$ of holomorphic orthogonal modular forms that are linear combinations Eisenstein series on the boundary. If $F \vert_I$ vanishes for every boundary component, we say that $F$ is a cusp form and we write $\S_\kappa(\Gamma(L))$ for the space of cusp forms.

For an easier exposition, we assume throughout the introduction that $L$ splits two hyperbolic planes over $\IZ$. Then the $0$-dimensional cusps of $\Gamma(L) \bs \IH_l$ are in bijective correspondence to the isotropic elements in $L' / L$ (which we denote by $\Iso(L' / L)$) up to $\pm 1$. For an appropriate automorphic form $f : \IH \to \IC[L' / L]$ of weight $k$ consider the additive Borcherds lift defined by \cite{Borcherds}
\begin{align}\label{eq:ThetaLift}
\Phi(Z, f) := \int_{\SL_2(\IZ) \bs \IH}^{\text{reg}} \langle f(\tau), \Theta_L(\tau, Z) \rangle v^k \frac{\mathrm{d}u \mathrm{d}v}{v^2},
\end{align}
where $\Theta_L$ is a certain Siegel theta function of weight $k$ in $\tau$ (see Section \ref{sec:SiegelthetaFunction}). In \cite{Kiefer} we have seen that the additive Borcherds lift $\Phi_{k, \beta}(Z, s)$ of a vector-valued non-holomorphic Eisenstein series $E_{k, \beta}(\tau, s), \beta \in \Iso(L' / L)$ of weight $k$ for the Weil representation $\rho_L$ defined by \cite{BruinierKuehn} (in \cite{BruinierKuehn} the dual Weil representation $\rho_L^*$ is considered) is a linear combination of orthogonal non-holomorphic Eisenstein series $\calG_{\kappa, \delta}(Z, s)$ of weight $\kappa = \frac{l}{2} - 1 + k$ with respect to $\Gamma(L)$ corresponding to a $0$-dimensional cusp $\delta \in \Iso(L' / L)$. To obtain holomorphic Eisenstein series one would try to evaluate this at the harmonic point $s = 0$. By calculating the Fourier expansion, one can split up $\Phi_{k,\beta}(Z) := \Phi_{k, \beta}(Z, 0)$ into a holomorphic part $\Phi_{k,\beta}^+(Z)$ and a non-holomorphic part $\Phi_{k,\beta}^-(Z)$. We will show

\begin{prop}[{see Proposition \ref{prop:HolomorphicBoundaryPart}}]
The holomorphic part $\Phi_{k,\beta}^+(Z)$ is given by
\begin{align*}
\Phi_{k,\beta}^{+}(Z)
&= \frac{\Gamma(\kappa) N_z^\kappa}{(-2 \pi i)^\kappa} \sum_{b \in \IZ / N_z \IZ} \delta_{\beta, \frac{bz}{N_z}} \zeta^b(\kappa) \\
&+ \sum_{\substack{\lambda \in K \\ q(\lambda) = 0 \\ (\lambda, Y) > 0 \\ \lambda \text{ primitive}}} \sum_{\substack{b\in \IZ / N_z \IZ \\ c \in \IZ / N_\lambda \IZ}} \delta_{\beta, \frac{c \lambda}{N_\lambda} - \frac{c (\lambda, \zeta)}{N_\lambda N_z} z + \frac{bz}{N}} \sum_{m = 1}^\infty \tilde{\sigma}_{\kappa - 1}^{c, b}(m) e\left(\frac{m (\lambda, Z - \frac{\zeta_K}{N_z})}{N_\lambda}\right) \\
&+ \sum_{\substack{\lambda \in K' \\ q(\lambda) > 0 \\ (\lambda, Y) > 0}} b(\lambda) e\left(\frac{(\lambda, Z - \frac{\zeta_K}{N_z})}{N_\lambda}\right)
\end{align*}
where $N_z$ and $N_\lambda$ are the levels of $z$ and $\lambda$,
$$\tilde{\sigma}_{\kappa - 1}^{c, b}(m) = \sum_{\substack{n \mid m \\ \frac{m}{n} \equiv c \bmod{N_\lambda}}} \sgn(n) n^{\kappa - 1} e\left(\frac{nb}{N_z}\right)$$
is a divisor sum over positive and negative divisors and $b(\lambda)$ are certain Fourier coefficients.
\end{prop}

Since only the terms with $q(\lambda) = 0$ in the Fourier expansion contribute to the restrictions to boundary components, one easily sees that the holomorphic part restricted to a boundary component is given by an Eisenstein series, see \cite[Theorem 4.2.3]{DiamondShurman} for their Fourier expansions. By inspecting the non-holomorphic part and using that vector-valued non-holomorphic Eisenstein series converge at $s = 0$ for $k > 2$ we obtain

\begin{thm}[{see Theorem \ref{thm:HolomorphicEisensteinSeriesk>2}}]
The non-holomorphic part $\Phi_{k,\beta}^-(Z)$ vanishes for $k > 2$ and hence $\Phi_{k,\beta}(Z) = \Phi_{k,\beta}^+(Z)$ is a holomorphic orthogonal modular form that is an Eisenstein series on the boundary. Moreover, every holomorphic orthogonal modular form that is a linear combination of Eisenstein series on the boundary is obtained as a theta lift (up to cusp forms).
\end{thm}

For $k = 0$ and hence singular weight $\kappa = \frac{l}{2} - 1$ this does not work immediately, since the non-holomorphic part does not vanish in general. Using the functional equation relating the values at $s$ and $1 - s$ we will instead inspect the behaviour at $s = 1$. We obtain

\begin{thm}[{see Theorem \ref{thm:ResidueAts=1}}]
The additive Borcherds lift $\Phi_{0,\beta}(Z, s)$ has a simple pole at $s = 1$ with residue given by
\begin{align*}
&\frac{\Gamma(\kappa) N_z^\kappa}{(-2 \pi i)^\kappa}\sum_{b \in \IZ / N_z \IZ} \res_{s = 1} c_{0,\beta}\left(\frac{bz}{N_z}, 0, s\right) \zeta^b(\kappa) \\
&+ \sum_{\substack{\lambda \in K' \\ q(\lambda) = 0 \\ (\lambda, Y) > 0}} e\left(- \frac{(\lambda, \zeta)}{N_z}\right) \sum_{b \in \IZ / N_z \IZ} \sum_{n \mid \lambda} n^{\kappa - 1} e\left(\frac{nb}{N_z}\right) \\
&\times \res_{s = 1} c_{0,\beta}\left(\frac{\lambda}{n} - \frac{(\lambda, \zeta)}{nN_z} z + \frac{bz}{N_z}, 0, s\right) e(\lambda, Z).
\end{align*}
where $c_{0,\beta}(0, 0, s)$ is a Fourier coefficients of the vector-valued non-holomorphic Eisenstein series $E_{0, \beta}(\tau, s)$. In particular, it is a holomorphic orthogonal modular form of singular weight that is an Eisenstein series on the boundary.
\end{thm}

In fact this is just the Borcherds lift of the invariant vector $\res_{s = 1} E_{0, \beta}(\tau, s)$ so that the theorem reads $\res_{s = 1} \Phi_{0, \beta}(Z, s) = \Phi(Z, \res_{s = 1} E_{0, \beta}(\cdot, s))$.

As in the higher weight case, it is natural to ask whether all holomorphic orthogonal modular forms of singular weight that are linear combinations of Eisenstein series on the boundary can be obtained in this way. Therefore we consider the adjoint theta lift $\Phi^*$ mapping holomorphic orthogonal modular forms of singular weight to modular forms of weight $0$. For an orthogonal modular form $F : \IH_l \to \IC$ it is given by
\begin{align}\label{eq:ThetaLiftAdjoint}
\Phi^*(\tau, F) := \int_{\Gamma(L) \bs \IH_l} F(Z) \Theta_L(\tau, Z) q(Y)^\kappa \frac{\mathrm{d}X \mathrm{d}Y}{q(Y)^l}.
\end{align}
We will show that if $F : \IH_l \to \IC$ is holomorphic of singular weight, then $\Phi^*(\tau, F)$ is a harmonic function of weight $0$ (see Lemma \ref{lem:AdjointThetaLiftHarmonic}). In Section \ref{sec:LiftingOrthogonal} we prove the following main theorem.

\begin{thm}[{see Theorem \ref{thm:AdjointThetaLiftInvariantVector}}]
For an orthogonal modular form of singular weight $\kappa = \frac{l}{2} - 1 > 0$ the theta lift $\Phi^*(\tau, F)$ is an invariant vector for the Weil representation given by
$$\frac{\Gamma(l/2)}{2(2 \pi)^{l / 2}} \sum_{\gamma \in \Iso(L' / L)} \sum_{\substack{\delta \in \Iso(L' / L) \\ \gamma = k_\delta \delta}} \zeta_+^{k_\delta}(l - \kappa) a_{F, \delta}(0) C(\delta) \frake_\gamma,$$
where $a_{F, \delta}(0)$ is the value of $F$ in the $0$-dimensional cusp corresponding to $\delta$ and $C(\delta)$ is a non-zero constant.
\end{thm}

As a corollary one obtains

\begin{cor}[{see Corollary \ref{cor:AdjointThetaLiftKernel}}]
Assume that $L$ splits two hyperbolic planes. The theta lift $\Phi^*(\tau, F)$ vanishes if and only if $F$ vanishes in every $0$-dimensional cusp.
\end{cor}

Since the theta lifts \eqref{eq:ThetaLift} and \eqref{eq:ThetaLiftAdjoint} are adjoint to each other we obtain

\begin{cor}[{see Corollary \ref{cor:ThetaLiftSurjective}}]
Assume that $L$ splits two hyperbolic planes. Then every holomorphic orthogonal modular form of singular weight for $\Gamma(L)$ that is an Eisenstein series on the boundary is the residue at $s = 1$ of some non-holomorphic Eisenstein series and the additive Borcherds lift is an isomorphism
$$\Inv(\IC[L' / L]) \to \M_\kappa^{\partial \Eis}(\Gamma(L)).$$
\end{cor}

If $L$ is a maximal lattice, then $\Inv(\IC[L' / L])$ is either $1$-dimensional (if $L$ is unimodular) or $0$-dimensional (if $L$ is not unimodular). Since maximal lattices of Witt rank $2$ always split two hyperbolic planes over $\IZ$, we obtain for singular weight $\kappa = \frac{l}{2} - 1$.

\begin{cor}[{see Corollary \ref{cor:ThetaLiftMaximalLattice}}]
If $L$ is a maximal lattice of Witt rank $2$, then the space $\M_\kappa^{\partial \Eis}(\Gamma(L))$ is either $1$-dimensional (if $L$ is unimodular) or $0$-dimensional (if $L$ is not unimodular). Moreover, if $\kappa = 2, 4, 6, 8, 10, 14$, i.e. $l = 6, 10, 14, 18, 22, 30$, then we have $\M_\kappa(\Gamma(L)) = \M_\kappa^{\partial \Eis}(\Gamma(L))$ and the space of holomorphic modular forms of singular weight is either $0$ or $1$-dimensional depending on $L$ being unimodular or not.
\end{cor}

If $L$ does not split two hyperbolic planes over $\IZ$, then we can still fully determine the image of the theta lift (see Corollary \ref{cor:ThetaLiftImage}).

\subsection*{Acknowledgment}

I would like to thank my advisor J.H. Bruinier for suggesting this topic as part of my doctoral thesis. Moreover, I would like to thank him for his support and our helpful discussions.

\section{Vector-Valued Non-Holomorphic Eisenstein Series}

We will now introduce the Weil representation and vector-valued modular forms. Therefore, let $\IH := \{\tau = u + iv \in \IC \mid v > 0 \}$ be the usual upper half-plane. For $z \in \IC$ we write $e(z) := e^{2 \pi i z}$ and we denote by $\sqrt{z} = z^{\frac{1}{2}}$ the principal branch of the square-root, i.e. $\arg(\sqrt{z}) \in (- \frac{\pi}{2}, \frac{\pi}{2}]$. 

\begin{defn}
We denote by $\Mp_2(\IR)$ the \emph{metaplectic cover} of $\SL_2(\IR)$. It is realized as pairs $(M, \phi)$, where $M \in \SL_2(\IR)$ and $\phi : \IH \to \IC$ is a holomorphic square root of $\tau \mapsto c \tau + d$. The product for $(M_1, \phi_1), (M_2, \phi_2) \in \Mp_2(\IR)$ is given by
$$(M_1, \phi_1(\tau))(M_2, \phi_2(\tau)) := (M_1 M_2, \phi_1(M_2 \tau) \phi_2(\tau)),$$
where $\left(\begin{smallmatrix} a & b \\ c & d\end{smallmatrix}\right) \tau = \frac{a \tau + d}{c \tau + d}$ is the usual action of $\SL_2(\IR)$.
\end{defn}
By $\Mp_2(\IZ)$ we denote the inverse image of $\SL_2(\IZ)$ under the covering map. It is generated by
$$T = \left(\begin{pmatrix}1 & 1 \\ 0 & 1\end{pmatrix}, 1\right) \quad \text{and} \quad S = \left(\begin{pmatrix}0 & -1 \\ 1 & 0\end{pmatrix}, \sqrt{\tau}\right).$$
We have the relation $S^2 = (ST)^3 = Z$, where $Z = \left(\left(\begin{smallmatrix}-1 & 0 \\ 0 & -1\end{smallmatrix}\right), i\right)$ is the standard generator of the center of $\Mp_2(\IZ)$. Furthermore we will write $\Gamma_\infty := \{\left(\begin{smallmatrix}1 & n \\ 0 & 1\end{smallmatrix}\right) \mid n \in \IZ \}$ and $\tilde{\Gamma}_\infty := \{\left(\left(\begin{smallmatrix}1 & n \\ 0 & 1\end{smallmatrix}\right), 1\right) \mid n \in \IZ \}$. For an even non-degenerate lattice $L$ of signature $(b^+, b^-)$ consider the group ring $\IC[L' / L]$ with standard basis $(\frake_\gamma)_{\gamma \in L' / L}$. For $\frakv = \sum_{\gamma \in L' / L} \frakv_\gamma \frake_\gamma \in \IC[L' / L]$ we write $\frakv^* := \sum_{\gamma \in L' / L} \frakv_\gamma \frake_{-\gamma}$. Moreover we write $\langle \cdot, \cdot \rangle$ for the standard inner product on $\IC[L' / L]$ which is anti-linear in the second variable. Write $\Iso(L' / L) \subseteq L' / L$ for the set of isotropic elements and denote the subspace of vectors that are supported on isotropic elements by $\Iso(\IC[L' / L])$. Moreover we introduce the notation $\frake_\gamma(\tau) := e(\tau) \frake_\gamma = e^{2 \pi i \tau} \frake_\gamma$.

\begin{defn}
The \emph{Weil representation} is the unitary representation $\rho_L$ of $\Mp_2(\IZ)$ on $\IC[L' / L]$ defined by
$$\rho_L(T) \frake_\gamma := \frake_\gamma(q(\gamma)) \quad \text{and} \quad \rho_L(S) \frake_\gamma := \frac{\sqrt{i}^{b^- - b^+}}{\sqrt{L' / L}} \sum_{\delta \in L' / L} \frake_\delta(-(\gamma, \delta)).$$
The Weil representation factors through a finite quotient of $\Mp_2(\IZ)$. The space of invariant vectors under the Weil representation is denoted by $\Inv(\IC[L' / L])$.
\end{defn}

A short calculation using orthogonality of characters shows for example $\rho_L(Z) \frake_\gamma = i^{b^- - b^+} \frake_{-\gamma}$. For $\beta, \gamma \in L' / L$ we define the coefficients
$$\rho_{\beta, \gamma}(M, \phi) := \langle \rho_L(M, \phi) \frake_\gamma, \frake_\beta \rangle \quad \text{and} \quad \rho_{\beta, \gamma}^{-1}(M, \phi) := \langle \rho_L^{-1}(M, \phi) \frake_\gamma, \frake_\beta \rangle.$$

\begin{thm}[{\cite[Proposition 1.6]{Shintani}}]\label{thm:Shintani}
For $M \in \SL_2(\IZ)$ the coefficient $\rho_{\beta, \gamma}(\tilde{M})$ is given by
$$\sqrt{i}^{(b^- - b^+)(1 - \sgn(d))} \delta_{\beta, a\gamma} e(ab q(\beta))$$
if $c = 0$ and by
$$\frac{\sqrt{i}^{(b^- - b^+) \sgn(c)}}{\lvert c \rvert^{(b^- + b^+)/2}\sqrt{\lvert L' / L \rvert}} \sum_{r \in L / cL} e\left(\frac{a(\beta + r, \beta + r) - 2 (\gamma, \beta + r) + d (\gamma, \gamma)}{2c}\right)$$
if $c \neq 0$.
\end{thm}

For a vector-valued function $f : \IH \to \IC[L' / L]$ we write $f_\gamma : \IH \to \IC$ for its components, i.e. $f = \sum_{\gamma \in L' / L} f_\gamma \frake_\gamma$. For $k \in \frac{1}{2} \IZ$ we define the \emph{Petersson slash operator} $f \mapsto f\vert_{k, L}(M, \phi)$ by
$$(f \vert_{k, L}(M, \phi))(\tau) = \phi(\tau)^{-2 k} \rho_L^{-1}(M, \phi) f(M \tau).$$
If $f : \IH \to \IC[L' / L]$ is smooth and invariant under the action of $T$, i.e. $f \vert_{k, L} T = f$, then we have a Fourier expansion
$$f(\tau) = \sum_{\gamma \in L' / L} \sum_{n \in \IZ + q(\gamma)} c(\gamma, n, v) \frake_\gamma(nu).$$

\begin{defn}
A function $f : \IH \to \IC[L' / L]$ is said to be a \emph{modular form} of weight $k$ with respect to the Weil representation if $f \vert_{k, L}(M, \phi) = f$ for all $(M, \phi) \in \Mp_2(\IZ)$. We call $f$ a \emph{holomorphic modular form} if $f$ is holomorphic with Fourier expansion
$$f(\tau) = \sum_{\gamma \in L' / L} \sum_{\substack{n \in \IZ + q(\gamma) \\ n \geq 0}} c(\gamma, n) \frake_\gamma(n\tau).$$
\end{defn}

Obviously, non-trivial modular forms only exist for weights with $2k + b^- - b^+ = 0 \bmod 2$.

Assume now that $b^+ - b^-$ is even and let $k \in \IZ$. Moreover set $\kappa = \frac{b^- - b^+}{2} + k$. Let $\beta \in \Iso(L' / L)$ and similar to \cite{BruinierKuehn} define the \emph{vector-valued non-holomorphic Eisenstein series} of weight $k$ by
$$E_{k, \beta}(\tau, s) = \frac{1}{2} \sum_{M \in \tilde{\Gamma}_\infty \bs \Mp_2(\IZ)} v^s \frake_\beta \vert_{k, L} M.$$
Observe that \cite{BruinierKuehn} consider the dual Weil representation $\rho_L$. For $\beta \in \Iso(L' / L)$ of order $N_\beta$ and a character $\chi : (\IZ / N_\beta \IZ)^\times$ we define
$$E_{k, \beta, \chi}(\tau, s) := \sum_{n \in (\IZ / N_\beta \IZ)^\times} \chi(n) E_{k, n \beta}(\tau, s).$$
More generally, for $\frakv \in \Iso(\IC[L' / L])$ we define
\begin{align*}
E_{k, \frakv}(\tau, s)
&= \frac{1}{2} \sum_{M \in \tilde{\Gamma}_\infty \bs \Mp_2(\IZ)} v^s \frakv \vert_{k, L} M \\
&= \sum_{\beta \in \Iso(L' / L)} \frakv_\beta E_{k, \beta}(\tau, s).
\end{align*}
We have $E_{k, \frakv^*} = (-1)^\kappa E_{k, \frakv}$ and a Fourier expansion of the form
\begin{align*}
E_{k, \frakv} &= (\frakv + (-1)^\kappa \frakv^*) v^s + \sum_{\gamma \in \Iso(L' / L)} c_{k, \frakv}(\gamma, 0, s) v^{1 - s - k} \\
&+ \sum_{\gamma \in L' / L} \sum_{\substack{n \in \IZ + q(\gamma) \\ n \neq 0}} c_{k, \frakv}(\gamma, n, s) \calW_s(4 \pi n y) \frake_\gamma(n u),
\end{align*}
where $\calW_s$ is a special Whittaker function. For the precise coefficients see \cite{BruinierKuehn}, \cite{Williams} (or \cite{BruinierKuss}, \cite{Scheithauer}, \cite{Schwagenscheidt} for the holomorphic case), we will not need them here. It can be easily seen that the vector-valued Eisenstein series are $\Mp_2(\IZ)$-translates of usual scalar-valued Eisenstein series. They are normalized such that their meromorphic continuation to $\IC$ are holomorphic in $s = 0$ and have a simple pole for $k = 0$ at $s = 1$ whose residue is an invariant vector. For an invariant vector $\frakv \in \Inv(\IC[L' / L])$ we have
$$E_k(\tau, s) \frakv = E_{k, \frakv}(\tau, s),$$
where $E_k(\tau, s)$ is the suitably normalized Eisenstein series for $\SL_2(\IZ)$. For $k > 2$ the series converges at $s = 0$ and defines a holomorphic Eisenstein series $E_{k, \frakv}(\tau) = E_{k, \frakv}(\tau, 0)$ with Fourier expansion
$$E_{k, \frakv}(\tau) = \frakv + (-1)^\kappa \frakv^* + \sum_{\gamma \in L' / L} \sum_{\substack{n \in \IZ + q(\gamma) \\ n > 0}} c_{k, \frakv}(\gamma, n) \frake_\gamma(n \tau).$$
For $k = 2$ they have a Fourier expansion of the form \cite{Miyake}, \cite{DiamondShurman}
$$E_{2, \frakv}(\tau) = \frakw v^{-1} + \frakv + (-1)^\kappa \frakv^* + \sum_{\gamma \in L' / L} \sum_{\substack{n \in \IZ + q(\gamma) \\ n > 0}} c_{2, \frakv}(\gamma, n) \frake_\gamma(n \tau).$$
Applying the lowering operator shows that $\frakw \in \Inv(\IC[L' / L])$ is an invariant vector. For $k = 0$ their residue at $s = 1$ yields an invariant vector and if $\frakv \in \Inv(\IC[L' / L])$ we obtain
$$\res_{s = 1} E_{0, \frakv}(\tau, s) = \res_{s = 1} E_0(\tau, s) \frakv,$$
in particular these residues span the space of invariants.

\section{Orthogonal Modular Forms}

From now on let $L$ be an even lattice of signature $(2, l)$ and let $V = L \otimes \IQ, V(\IR) = V \otimes \IR$. Write $P(V(\IC))$ for the corresponding \emph{projective space}. For elements $Z_L = X_L + i Y_L \in V(\IC) \setminus \{0\}$ we write $[Z_L]$ for the canonical projection onto $P(V(\IC))$. The subset
$$\calK = \{[Z_L] \in P(V(\IC)) \mid (Z_L, Z_L) = 0, (Z_L, \overline{Z_L}) > 0 \}$$
is the hermitian symmetric domain associated to $O(V)$. It has two connected components which are interchanged by $[Z_L] \mapsto [\overline{Z_L}]$. We choose one of them and call it $\calK^+$. The action of $O(V(\IR))$ on $V(\IR)$ induces an action on $\calK$. Let $O^+(V(\IR))$ be the subgroup which preserves the connected components of $\calK$ and let
$$\tilde{\calK}^+ = \{Z_L \in V(\IC) \setminus \{0\} \mid [Z_L] \in \calK^+ \}$$
be the preimage of $\calK^+$ under the projection. Write $\Iso_0(L)$ for the set of primitive isotropic elements of $L$. For $z \in \Iso_0(L)$ and $z' \in L'$ with $(z, z') = 1$ write $K = L \cap z^\perp \cap z'^\perp$. Let $d \in \Iso_0(K)$, $d' \in K'$ with $(d, d') = 1$ and $D = K \cap d^\perp \cap d'^\perp$. Moreover let $\tilde{z} = z' - q(z') z$ and $\tilde{d} = d' - q(d') d$. Let $d_3, \ldots, d_l$ be a basis of $D$. Then $\tilde{d}, d, d_3, \ldots, d_l$ is a basis of $K \otimes \IR$. We define the orthogonal upper half plane as
$$\IH_l := \{Z = X + iY \in W(\IC) = K \otimes \IC \mid q(Y) > 0, (Y, d) > 0\}.$$
We will write $Z = z_1 \tilde{d} + z_2 d + Z_D$ with $Z_D \in D \otimes \IC$ and analogously for $X, Y \in W(\IR)$. For $Z \in \IH_l$ we define
$$Z_L := Z - q(Z) z + \tilde{z}.$$
If the component $\calK^+$ is chosen properly, this yields a biholomorphic map $Z \mapsto [Z_L]$. By setting $i C = \IH_l \cap i (K \otimes \IR)$ we see that $\IH_l = K \otimes \IR + iC$ is a \emph{tube domain}. The action of $O^+(V(\IR))$ on $\calK^+$ induces an action of $O^+(V(\IR))$ on $\IH_l$. Let
$$j : O^+(V) \times \IH_l \to \IC^\times, \quad j(\sigma, Z) := (\sigma(Z_L), z)$$
be the \emph{factor of automorphy}, so that we have
$$j(\sigma, Z)(\sigma Z)_L = \sigma(Z_L)$$
and the cocycle relation
$$j(\sigma_1 \sigma_2, Z) = j(\sigma_1, \sigma_2 Z) j(\sigma_2, Z).$$
According to \cite[Lemma 3.20]{BruinierChern} we have
$$q(\Im(\sigma Z)) = \frac{q(\Im(Z))}{\lvert j(\sigma, Z) \rvert^2}.$$

\begin{defn}
A function $F : \tilde{\calK}^+ \to \IC$ is called \emph{modular form} of weight $\kappa \in \IZ$ with respect to $\Gamma \subseteq \Gamma(L)$ if it satisfies
\begin{enumerate}
\item[(i)] $F(t Z_L) = t^{-\kappa} F(Z_L)$ for all $t \in \IC^\times$.
\item[(ii)] $F(\sigma Z_L) = F(Z_L)$ for all $\sigma \in \Gamma$.
\end{enumerate}
\end{defn}

For a modular form $F : \tilde{\calK}^+ \to \IC$ of weight $\kappa \in \IZ$ with respect to $\Gamma$ define
$$F_z : \IH_l \to \IC, \quad F_z(Z) := F(Z_L) = F(Z - q(Z) z + \tilde{z}).$$
Then $F_z$ satisfies
$$F_z(\sigma Z) = j(\sigma, Z)^\kappa F_z(Z)$$
for all $\sigma \in \Gamma$ and we have a bijective correspondence between modular forms and functions with this transformation property. If we define the weight $\kappa$ slash operator $F_z \vert_\kappa \sigma$ for $\sigma \in O^+(V)$ by
$$(F_z \vert_\kappa \sigma)(Z) := j(\sigma, Z)^{-\kappa} F_z(\sigma Z),$$
then the modular forms on $\IH_l$ are exactly the functions that are invariant under the slash operator for all $\sigma \in \Gamma$. If $F_z$ is holomorphic on $\IH_l$ it has a Fourier expansion of the form
$$F_z(Z) = \sum_{\lambda \in K'} a_z(\lambda) e(\lambda, Z),$$
where $e(\lambda, Z) := e((\lambda, Z)) = e^{2 \pi i (\lambda, Z)}$.

\begin{defn}
We say that a modular form $f$ is a \emph{holomorphic modular form} if $f$ is holomorphic and for all $0$ dimensional cusps $z \in \Iso_0(L)$ we have $a_z(\lambda) = 0$ for $\lambda \notin \overline{C}$. We write $\M_\kappa(\Gamma)$ for the space of modular forms.
\end{defn}

\begin{rem}
For $l \geq 4$ the weight $\kappa = \frac{l}{2} - 1$ is called \emph{singular weight}. By the theory of singular weights, it is the smallest positive weight such that there are holomorphic modular forms for $\Gamma(L)$. Moreover, if
$$F_z(Z) = \sum_{\lambda \in K'} a_z(\lambda) e(\lambda, Z)$$
is a holomorphic modular form of singular weight, then $a_z(\lambda) \neq 0$ implies $\lambda \in C$.
\end{rem}

\begin{defn}
For $t > 0$ we define the \emph{Siegel domain} $\calS_t$ as the set of $Z = X + iY \in \IH_l$ satisfying
\begin{align*}
x_1^2 + x_2^2 + \lvert q(X_D) \rvert &< t^2, \\
1 / t &< y_1, \\
y_1^2 &< t^2 q(Y), \\
\lvert q(Y_D) \rvert &< t^2 y_1^2.
\end{align*}
The set of $Y \in C$ satisfying the last three inequalities is denoted by $\calR_t$. 
\end{defn}

We have the following

\begin{prop}[{\cite[Proposition 4.10]{BruinierChern}}]
Let $\Gamma \subseteq \Gamma(L)$ be a subgroup of finite index.
\begin{enumerate}
\item[(i)] For any $t > 0$ and any $g \in O^+(V)$ the set
$$\{\sigma \in \Gamma \mid \sigma g \calS_t \cap \calS_t \neq \emptyset \}$$
is finite.
\item[(ii)] There exists a $t > 0$ and finitely many $g_1, \ldots, g_n \in O^+(V)$ such that for
$$\calS = g_1 \calS_t \cup \ldots \cup g_n \calS_t$$
we have $\Gamma \calS = \IH_l$.
\end{enumerate}
\end{prop}

The invariant volume element on $\IH_l$ is given by
$$\frac{\mathrm{d}X \mathrm{d}Y}{q(Y)^l}.$$
The previous proposition means in particular, that for a $\Gamma$-invariant measurable function $F : \IH_l \to \IC$ we have $F \in L^p(\IH_l / \Gamma)$ if and only if $\int_\calS \lvert F(Z) \rvert^p \frac{\mathrm{d}X \mathrm{d}Y}{q(Y)^l} < \infty$. This again is the case if for every choice of $z, z', d, d', t > 0$ the integral $\int_{\calS_t} \lvert F(Z) \rvert^p \frac{\mathrm{d}X \mathrm{d}Y}{q(Y)^l}$ is finite. We will need the following estimates.

\begin{lem}[{\cite[Lemma 4.13]{BruinierChern}}]\label{lem:SiegelDomainInequality}
Let $t > 0$. Then there exists $\varepsilon > 0$ such that for all $\lambda = (\lambda_1, \lambda_2, \lambda_D) \in K'$ and $Y \in \calR_t$ we have
$$\frac{(\lambda, Y)^2}{Y^2} - q(\lambda) \geq \varepsilon (y_2^2 \lambda_1^2 / 2 + y_1^2 \lambda_2^2 / 2 + Y^2 \lambda_D^2) / Y^2.$$
Using the inequality
$$\frac{y_1 y_2}{1 + t^4} < q(Y) < y_1 y_2$$
this yields
$$\frac{(\lambda, Y)^2}{Y^2} - q(\lambda) > \varepsilon (y_2 / y_1 \lambda_1^2 / 2 + y_1 / y_2 \lambda_2^2 / 2 + \lambda_D^2).$$
\end{lem}

The following lemma is a slight generalisation of \cite[Lemma 4.11, 4.12]{BruinierChern}

\begin{lem}\label{lem:SiegelDomainIntegrable}
Suppose $p_2 + p < l - 1$ and $p_1 + p_2 + 2p < l$. Then
$$\int_{\calS_t} y_1^{p_1} y_2^{p_2} q(Y)^p \frac{\mathrm{d}X \mathrm{d}Y}{q(Y)^l} < 0.$$
\end{lem}

\section{Siegel Operator}

For an isotropic line $I \subseteq V(\IR)$ generated by some isotropic vector $\lambda \in I$ the corresponding point $[\lambda]$ is in the closure of $\calK^+$ in $\calN$. To see this, take two sequences $x_n, y_n \in V(\IR)$ with positive norm which are orthogonal and converge to $\lambda$. Then $[x_n + i y_n] \in \calK^+$ with limit $[(1 + i)\lambda] = [\lambda] \in \calN$ for $n \to \infty$.

\begin{defn}
A boundary point of $\calK^+$ of the form $[\lambda] \in \calN$ which is represented by a real isotropic line is called \emph{special boundary point}. A set consisting of one special boundary points is called \emph{$0$-dimensional boundary component}. A non-special boundary point is called \emph{generic boundary point}.
\end{defn}

Let now $I \subseteq V(\IR)$ be an isotropic plane and consider the set of all boundary points which can be represented by elements of $I \otimes \IC$.

\begin{defn}
For an totally isotropic plane $I \subseteq V(\IR)$ the set of all generic boundary points which can be represented by an element of $I \otimes \IC$ is called \emph{$1$-dimensional boundary component} attached to $I$. By a \emph{boundary component} we mean a $0$-dimensional or $1$-dimensional boundary component.
\end{defn}

\begin{lem}
The $1$-dimensional boundary components are isomorphic to usual upper half-planes. Moreover, there is a bijective correspondence between boundary components and non-zero isotropic subspaces of $V(\IR)$.
\end{lem}

\begin{proof}
Let $I \subseteq V(\IR)$ be an isotropic plane. Take a basis $z, d$ of $I$ and consider $\tilde{z}, \tilde{d}$ isotropic such that $(z, \tilde{z}) = (d, \tilde{d}) = 1$ and all other products vanish. We will use the shorthand notation $(z_1, z_2, z_3, z_4)$ for $z_1 \tilde{z} + z_2 z + z_3 \tilde{d} + z_4 d$. Then the elements of $I$ have the form $(0, z_2, 0, z_4)$. Assume that this is not a multiple of a real point, i.e. it is not a special boundary points. Then $z_2 \neq 0 \neq z_4$ and we can normalize it such that $z_4 = 1$, i.e. we have a point of the form $(0, \tau, 0, 1)$ for some $\tau \in \IC \setminus \IR$. making suitable choices of the basis and $\calK^+$ we can assume that $[1, 1, i, i] \in \calK^+$. Then we have an embedding
$$\IH \times \IH \to \calK^+, \quad (\tau_1, \tau_2) \mapsto [1, - \tau_1 \tau_2, \tau_1, \tau_2].$$
In the projective space we have
$$\lim_{t \to \infty} [1, - \tau it, \tau, it] = \lim_{t \to \infty} \left[\frac{1}{it}, -\tau, \frac{\tau}{it}, 1\right] = [0, -\tau, 0, 1].$$
Thus the point $[0, -\tau, 0, 1]$ is in the boundary of $\calK^+$ if and only if $v = \Im(\tau) > 0$. The set of all boundary points represented by elements in $I \otimes \IC$ can be identified with $\IH \cup \IR \cup \infty$. In particular, if we let $\tau = iv$ with $v \to \infty$, we obtain the special boundary point $[z]$.
\end{proof}

\begin{defn}
A boundary component is called \emph{rational boundary component} if the corresponding isotropic subspace is defined over $\IQ$. Write $\IH_l^*$ for the union of $\IH_l \simeq \calK^+$ with all rational boundary components. Then the rational orthogonal group $O^+(V) := O^+(V(\IR)) \cap O(V)$ acts on $\IH_l^*$.
\end{defn}

By the theory of Baily-Borel (see \cite{BailyBorel}, \cite{BorelJi}), there is a topology on $\IH_l^*$ such that for congruence subgroups $\Gamma \subseteq O^+(V)$ the quotient $X_\Gamma = \IH_l^* / \Gamma$ carries the structure of a projective variety, which contains $\IH_l / \Gamma$ as a Zariski open subvariety. 

Let now $I \subseteq L \otimes \IQ$ be an isotropic plane. Let $z \in L \cap I$ be primitive and $z' \in L$ with $(z, z') = 1$. Let $K = L \cap z^\perp \cap z'^\perp$ and take $d \in K \cap I$ primitive, $d' \in K'$ with $(d, d') = 1$ and let $D = K \cap d^\perp \cap d'^\perp$. As before, write $\tilde{z} = z' - q(z') z, \tilde{d} = d' - q(d') d$ and let $d_3, \ldots, d_l$ be a basis of $D$. Then we obviously have $I = \langle z, d \rangle$. Recall the orthogonal upper half plane corresponding to $z$
$$\IH_l := \{Z = X + iY \in K \otimes \IC \mid q(Y) > 0, (Y, d) > 0 \}$$
and that we write $Z = z_1 \tilde{d} + z_2 d + Z_D$. Let $C \subseteq K \otimes \IR$ be the positive cone such that
$$\IH_l = K \otimes \IR + i C$$
and write $\overline{C}$ for its closure. For $\tau \in \IH$ and $t \in \IR_{>0}$ we have $\tau \tilde{d} + it d \in \IH_l$, which corresponds to $[\tilde{z} - \tau it z + \tau \tilde{d} + it d] \in \calK^+$. For a function $F : \IH_l \to \IC$ we define the Siegel operator corresponding to the boundary component $I$ as (if it exists)
$$F \vert_I : \IH \to \IC, \quad \tau \mapsto \lim_{t \to \infty} F(\tau \tilde{d} + it d).$$
For $\lambda = (\lambda_1, \lambda_2, \lambda_D) \in K'$ we have
$$(\lambda, \tau \tilde{d} + itd) = it \lambda_1 + \tau \lambda_2.$$
Assume that $F$ has a Fourier expansion of the form
$$F(Z) = \sum_{\substack{\lambda \in K' \\ \lambda \in \overline{C}}} a(\lambda) e(\lambda, Z),$$
then the Siegel operator exists and we have
\begin{align*}
F \vert_I(\tau)
&= \lim_{t \to \infty} \sum_{\substack{(\lambda_1, \lambda_2, \lambda_D) \in K' \\ (\lambda_1, \lambda_2, \lambda_D) \in \overline{C}}} a(\lambda_1, \lambda_2, \lambda_D) e(\lambda_1 \tilde{d} + \lambda_2 d + \lambda_D, \tau d' + itd) \\
&= \lim_{t \to \infty} \sum_{\substack{(\lambda_1, \lambda_2, \lambda_D) \in K' \\ (\lambda_1, \lambda_2, \lambda_D) \in \overline{C}}} a(\lambda_1, \lambda_2, \lambda_D) \exp(-2 \pi t \lambda_1) e(\lambda_2 \tau) \\
&= \sum_{\substack{(0, \lambda_2, 0) \in K' \\ \lambda_2 \geq 0}} a(0, \lambda_2, 0) e(\lambda_2 \tau) \\
&= \sum_{m = 0}^\infty a(0, m / N_d, 0) e(m \tau / N_d).
\end{align*}
The value in the $0$-dimensional cusp $z$ is given by the constant term of the Fourier expansion, i.e.
$$a(0) = \lim_{t \to \infty} F(t Z)$$
for arbitrary $Z \in \IH_l$.

Let now $F : \tilde{\calK}^+ \to \IC$ be a holomorphic modular form of weight $\kappa$ for some congruence subgroup $\Gamma$. Then its restriction to the boundary component $F_z \vert_I : \IH \to \IC$ is a holomorphic modular form of weight $\kappa$ for some subgroup of $\SL_2(\IQ)$. One easily sees that for the constant Fourier coefficient we have $a_{-z}(0) = (-1)^\kappa a_z(0)$.

\begin{defn}
We call a holomorphic modular form $F : \IH_l \to \IC$ of weight $\kappa$ for some congruence subgroup $\Gamma$ a \emph{cusp form} if $F \vert_I$ vanishes identically for all boundary components $I$. We say that $F$ is an \emph{Eisenstein series on the boundary} if $F \vert_I$ is an Eisenstein series for all boundary components $I$. We write $\S_\kappa(\Gamma)$ for the space of cusp forms and $\M_\kappa^{\partial \Eis}(\Gamma)$ for the space of holomorphic modular forms that are Eisenstein series on the boundary. In particular we have $\S_\kappa(\Gamma) \subseteq \M_\kappa^{\partial \Eis}(\Gamma)$.
\end{defn}

\begin{rem}
By the theory of singular weights, there are no cusp forms of singular weight. In particular, a holomorphic modular form of singular weight is completely determined by its values at the boundary and a holomorphic modular form of singular weight that is an Eisenstein series on the boundary is completely determined by its values in the $0$-dimensional cusps.
\end{rem}

\begin{rem}
If a holomorphic modular form $F$ is an Eisenstein series on the boundary, then its restrictions to the boundary are fully determined by the values in the $0$-dimensional cusps, i.e. the constant Fourier coefficients.
\end{rem}

\begin{defn}
For $\Gamma = \Gamma(L)$ recall the map
$$\pi_L : \{\text{$0$-dimensional cusps of } \Gamma(L) \bs \IH_l \} \simeq \Gamma(L) \bs \Iso_0(L')  \to \Iso(L' / L).$$
We write $\M_\kappa^{\pi}(\Gamma(L))$ for the subspace of $\M_\kappa^{\partial \Eis}(\Gamma(L))$ that consists of holomorphic modular forms whose values in the $0$-dimensional cusps only depend on their image in $L' / L$. In particular, if $\delta = -\delta \in L' / L$, the value in the $0$-dimensional cusps corresponding to $\delta$ vanish if $\kappa$ is odd. We have $\S_\kappa(\Gamma(L)) \subseteq \M_\kappa^\pi(\Gamma(L))$ and if $\pi_L$ is injective we have $\M_\kappa^\pi(\Gamma(L)) = \M_\kappa^{\partial \Eis}(\Gamma(L))$.
\end{defn}

We want to mention the following result which is sometimes called Eichler criterion.

\begin{lem}[{\cite[Lemma 4.4]{FreitagHermann}}]
If $L$ splits two hyperbolic planes, then $\pi_L$ is bijective.
\end{lem}

\section{Differential Operators}\label{sec:WeightLaplace}

Let $M$ be a hermitian manifold and $E$ a holomorphic vector bundle. We write $\calE(E) = \calE(M, E)$ for the space of (global) sections of $E$ and, more generally, for $U \subseteq M$ open, we write $\calE(U, E)$ for the space of sections over $U$. Assume now that $E$ carries an hermitian metric. Then the hermitian metric on $E$ induces a hermitian metric on $E$-valued differential forms $\calA^*(M, E)$ locally given by
$$(\varphi \otimes \sigma, \varphi' \otimes \sigma') = (\varphi, \varphi') (\sigma, \sigma')$$
on decomposable forms, where the first factor is the hermitian metric on differential forms coming from the hermitian metric on $M$ and the second factor is the hermitian metric on the vector bundle $E$. Moreover, the hermitian metric on $E$ together with the Hodge-$*$-operator $* : \calA^*(M) \to \calA^*(M)$ induces a Hodge-$*$-operator
$$\overline{*}_E : \calA^*(M, E) \to \calA^*(M, E^*).$$
The evaluation map $E \otimes E^* \to \IC$ induces a wedge product $\wedge$ which satisfies
$$\alpha \wedge \overline{*}_E \beta = (\alpha, \beta) \vol,$$
where $\vol$ is the volume form induced from the hermitian metric on $M$. Moreover we have a hermitian inner product on $E$-valued differential forms given by
$$(\alpha, \beta)_2 = \int_M \alpha \wedge \overline{*}_E \beta = \int_M (\alpha, \beta) \vol.$$
Let $\overline{\partial}^*_E = \overline{*}_{E^*} \overline{\partial}_{E^*} \overline{*}_E$ and $\Delta_E = \overline{\partial}^*_E \overline{\partial}_E$ be the Laplace operator on sections of $E$. We will need the following

\begin{thm}[{\cite[Section 3, Example (B)]{Chernoff}}]\label{thm:LaplaceSelfAdjoint}
Let $M$ be a complete connected hermitian manifold and let $E$ be an hermitian vector bundle. If $u, v$ are smooth square integrable sections of $E$ such that $\Delta_E u, \Delta_E v$ are also square integrable, then
$$(\Delta_E u, v) = (u, \Delta_E v).$$
\end{thm}

Let now $L$ be an even lattice of signature $(2, l)$. For $z \in \Iso_0(L)$ and $z' \in L'$ with $(z, z') = 1$ write $K = L \cap z^\perp \cap z'^\perp$. Let $b_1, \ldots, b_l$ be a basis of $K \otimes \IR$ such that
$$q(y_1 b_1 + \ldots + y_l b_l) = y_1^2 - y_2^2 - \ldots \ldots y_l^2.$$
If $Z = z_1 b_1 + z_2 b_2 + z_3 b_3 + \ldots z_l b_l \in K \otimes \IC$, we write $Z = (z_1, \ldots, z_l)$ and similarly $X = (x_1, \ldots x_l), Y = (y_1, \ldots, y_l)$ if $Z = X + i Y$ with $X, Y \in K \otimes \IR$. Denote by $\IH_l = K \otimes \IR + i C$ the corresponding tube domain model, where
$$C = \{Y = (y_1, \ldots, y_l) \in K_z \otimes \IR \mid y_1 > 0, q(Y) > 0\}.$$
The K\"ahler form is given by
$$w = -\frac{i}{2} \partial \overline{\partial} \log(q(Y))) = \frac{i}{2} \sum_{i, j} h_{ij}(Z) \mathrm{d}z_i \wedge \mathrm{d}\overline{z}_j,$$
where $h(Z) = h(Y) = (h_{ij})_{1 \leq i, j \leq l}$ is the associated hermitian form given by
\begin{align*}
\frac{1}{q(Y)^2}
\begin{pmatrix}
y_1^2 & -y_1 y_2 & -y_1 y_3 & \hdots & -y_1 y_l \\
-y_1 y_2 & y_2^2 & y_2 y_3 & \hdots & y_2 y_l \\
-y_1 y_3 & y_2 y_3 & y_3^2 & \ddots & \vdots \\
\vdots & \vdots & \ddots & \ddots & y_{l-1} y_l \\
-y_1 y_l & y_2 y_l & \hdots & y_{l-1} y_l & y_l^2
\end{pmatrix}
+ \frac{1}{2q(Y)}
\begin{pmatrix} -1 &  &  & \\
& 1 &  &  &  \\
&  & \ddots & &  \\
&  &  & \ddots &  \\
&  &  &  & 1
\end{pmatrix}
\end{align*}
and its inverse is
\begin{align*}
h^{-1}(Y) = 4
\begin{pmatrix}
y_1^2 & y_1 y_2 & y_1 y_3 & \hdots & y_1 y_l \\
y_1 y_2 & y_2^2 & y_2 y_3 & \hdots & y_2 y_l \\
y_1 y_3 & y_2 y_3 & y_3^2 & \ddots & \vdots \\
\vdots & \vdots & \ddots & \ddots & y_{l-1} y_l \\
y_1 y_l & y_2 y_l & \hdots & y_{l-1} y_l & y_l^2
\end{pmatrix}
+ 2 q(Y)
\begin{pmatrix} -1 &  &  & \\
& 1 &  &  &  \\
&  & \ddots & &  \\
&  &  & \ddots &  \\
&  &  &  & 1
\end{pmatrix}.
\end{align*}
Its determinant is $\det(h) = \frac{1}{2^{l} q(Y)^l}$ and hence the volume form is given by
\begin{align*}
\omega_g
&= \frac{1}{(4i q(Y))^l} \mathrm{d}z_1 \wedge \mathrm{d}\overline{z}_1 \wedge \ldots \wedge \mathrm{d}z_l \wedge \mathrm{d}\overline{z}_l \\
&= \frac{1}{2^l q(Y)^{l}} \mathrm{d}x_1 \wedge \mathrm{d}y_1 \wedge \ldots \wedge \mathrm{d}x_l \wedge \mathrm{d}y_l.
\end{align*}
We write
$$\widehat{\mathrm{d}z_i} = \mathrm{d}z_1 \wedge \mathrm{d}\overline{z}_1 \wedge \ldots \wedge \mathrm{d}\overline{z_{i-1}} \wedge \mathrm{d}\overline{z}_i \wedge \ldots \wedge \mathrm{d}z_l \wedge \mathrm{d}\overline{z}_l$$
and similarly for $\widehat{\mathrm{d}\overline{z}_i}$. Then we have
$$\mathrm{d}z_i \wedge \widehat{\mathrm{d}z_i} = (4i q(Y))^l \omega_g, \quad \mathrm{d}\overline{z}_i \wedge \widehat{\mathrm{d}\overline{z}_i} = -(4i q(Y))^l \omega_g.$$
Now the Hodge-$\overline{*}$-operator is defined by the equality
$$\alpha \wedge \overline{*} \beta = \langle \alpha, \beta \rangle \omega_g,$$
and thus we have
$$\overline{*} \mathrm{d}\overline{z}_i = -\frac{1}{2} \sum_{j = 1}^l \frac{h^{ji}(Y)}{(4 i q(Y))^l} \widehat{\mathrm{d}\overline{z}_j}.$$
Modular forms of weight $\kappa$ form a line bundle $\calL_\kappa$ which carries a hermitian metric given by the Petersson metric, i.e. $F(Z) \overline{G(Z)} q(Y)^\kappa$ on the fiber. The dual bundle $\calL_\kappa^*$ can be identified using the hermitian metric with the line bundle $\calL_{-\kappa}$ of modular forms of weight $-\kappa$ and the mapping $F(Z) \mapsto q(Y)^\kappa \overline{F(Z)}$ defines an anti-linear bundle isomorphism. This gives us the Hodge-$\overline{*}$-operator
$$\overline{*}_{\kappa} : \calA^{p, q}(\Gamma \bs \IH_l, \calL_\kappa) \to \calA^{n - p, n - q}(\Gamma \bs \IH_l, \calL_{-\kappa}), \quad \overline{*}_{\kappa}(\phi \otimes F) := (\overline{*} \phi) \otimes (q(Y)^\kappa \overline{F}).$$
The weight $\kappa$ Laplace operator is then given by
$$\Omega_\kappa = \overline{*}_{-\kappa} \overline{\partial} \overline{*}_\kappa \overline{\partial}.$$

\begin{thm}
For a modular form $F$ of weight $\kappa$ the weight $\kappa$ Laplace operator is given by
\begin{align*}
\Omega_\kappa F(Z)
&= \frac{q(Y)^{l-\kappa}}{2} \sum_{j = 1}^l \sum_{i = 1}^l \frac{\partial}{\partial z_j}\left(h^{ji}(Y) q(Y)^{\kappa - l} \frac{\partial F(Z)}{\partial \overline{z}_i} \right) \\
&= 2 \sum_{j = 1}^l \sum_{i = 1}^l y_i y_j \frac{\partial^2 F(Z)}{\partial z_j \partial \overline{z}_i} - q(Y) \left( \frac{\partial^2 F(Z)}{\partial z_1 \partial \overline{z}_1} - \sum_{i = 2}^l \frac{\partial^2 F(Z)}{\partial z_i \partial \overline{z}_i} \right) - i \kappa \sum_{i = 1}^l y_i \frac{\partial F(Z)}{\partial \overline{z}_i}.
\end{align*}
\end{thm}

\begin{proof}
We have
\begin{align*}
\overline{*}_\kappa \overline{\partial} F(Z) &= \sum_{i = 1}^l q(Y)^k \frac{\partial \overline{F(Z)}}{\partial z_i} \overline{*} \mathrm{d}\overline{z}_i \\
&= -\frac{1}{2 (4i)^l} \sum_{j = 1}^l \sum_{i = 1}^l h^{ji}(Y) q(Y)^{\kappa - l} \frac{\partial \overline{F(Z)}}{\partial z_i} \widehat{\mathrm{d}\overline{z}_j}.
\end{align*}
Applying $\overline{*}_{-\kappa} \overline{\partial}$ yields
\begin{align*}
\Omega_\kappa F(Z)
&= -\frac{1}{2 (4i)^l} \sum_{j = 1}^l \sum_{i = 1}^l \overline{*}_{-\kappa} \overline{\partial} \left(h^{ji}(Y) q(Y)^{\kappa - l} \frac{\partial \overline{F(Z)}}{\partial z_i} \widehat{\mathrm{d}\overline{z}_j}\right) \\
&= -\frac{1}{2 (4i)^l} \sum_{j = 1}^l \sum_{i = 1}^l \overline{*}_{-\kappa} \frac{\partial}{\partial \overline{z}_j}\left(h^{ji}(Y) q(Y)^{\kappa - l} \frac{\partial \overline{F(Z)}}{\partial z_i} \right) \mathrm{d}\overline{z}_j \wedge \widehat{\mathrm{d}\overline{z}_j} \\
&= \frac{q(Y)^l}{2} \sum_{j = 1}^l \sum_{i = 1}^l \overline{*}_{-\kappa} \frac{\partial}{\partial \overline{z}_j}\left(h^{ji}(Y) q(Y)^{\kappa - l} \frac{\partial \overline{F(Z)}}{\partial z_i} \right) \omega_g \\
&= \frac{q(Y)^{l-\kappa}}{2} \sum_{j = 1}^l \sum_{i = 1}^l \frac{\partial}{\partial z_j}\left(h^{ji}(Y) q(Y)^{\kappa - l} \frac{\partial F(Z)}{\partial \overline{z}_i} \right) \omega_g \\
&= \frac{1}{2} \sum_{j = 1}^l \sum_{i = 1}^l h^{ij}(Y) \frac{\partial^2 F(Z)}{\partial z_j \partial \overline{z}_i} + \frac{q(Y)^{l - \kappa}}{2} \sum_{j = 1}^l \sum_{i = 1}^l \left(\frac{\partial}{\partial z_j} h^{ij}(Y) q(Y)^{\kappa - l}\right) \frac{\partial F(Z)}{\partial \overline{z}_i}.
\end{align*}
A short calculation yields
\begin{align*}
\frac{\partial q(Y)}{\partial y_1} = 2y_1, \quad \frac{\partial q(Y)}{\partial y_j} = -2y_j, j > 1, \quad \frac{\partial h^{ij}(Y)}{\partial y_1} = 4 y_i.
\end{align*}
Thus, for $i = 1$ we have
\begin{align*}
&\sum_{j = 1}^l \frac{\partial}{\partial z_j} h^{ij}(Y) q(Y)^{\kappa - l} \\
&= -\frac{i}{2} \sum_{j = 1}^l \left(q(Y)^{\kappa - l} \frac{\partial h^{ij}(Y)}{\partial y_j} + (\kappa - l) h^{ij}(Y) q(Y)^{\kappa - l - 1} \frac{\partial q(Y)}{\partial y_j} \right) \\
&= -\frac{i}{2} \left(4q(Y)^{\kappa - l} l y_1 + 2(\kappa - l) q(Y)^{\kappa - l - 1} \left((4 y_1^2 - 2 q(Y)) y_1 - 4 y_1 \sum_{j = 2}^l y_j^2\right) \right) \\
&= -\frac{i}{2} \left(4q(Y)^{\kappa - l} l y_1 + 2 y_1 (\kappa - l) q(Y)^{\kappa - l - 1} \left(4 y_1^2 - 2 q(Y) + 4 \sum_{j = 2}^l y_j^2\right) \right) \\
&= -2 i \kappa q(Y)^{\kappa - l} y_1
\end{align*}
and similarly for $i > 1$
\begin{align*}
&\sum_{j = 1}^l \frac{\partial}{\partial z_j} h^{ij}(Y) q(Y)^{\kappa - l} \\
&= -\frac{i}{2} \left(4q(Y)^{\kappa - l} l y_i + 2(\kappa - l) q(Y)^{\kappa - l - 1} \left(4 y_i y_1^2 - 4 y_i \sum_{j = 2}^l y_j^2 - 2q(Y) y_i\right) \right) \\
&= -2 i \kappa q(Y)^{\kappa - l} y_i
\end{align*}
and hence
\begin{align*}
\Omega_\kappa F(Z)
&= \frac{1}{2} \sum_{j = 1}^l \sum_{i = 1}^l h^{ij}(Y) \frac{\partial^2 F(Z)}{\partial z_j \partial \overline{z}_i} - i \kappa \sum_{i = 1}^l y_i \frac{\partial F(Z)}{\partial \overline{z}_i} \\
&= 2 \sum_{j = 1}^l \sum_{i = 1}^l y_i y_j \frac{\partial^2 F(Z)}{\partial z_j \partial \overline{z}_i} - q(Y) \left( \frac{\partial^2 F(Z)}{\partial z_1 \partial \overline{z}_1} - \sum_{i = 2}^l \frac{\partial^2 F(Z)}{\partial z_i \partial \overline{z}_i} \right) - i \kappa \sum_{i = 1}^l y_i \frac{\partial F(Z)}{\partial \overline{z}_i}.
\end{align*}
\end{proof}

\begin{rem}\label{rem:Shaul}
There is an ad hoc definition of the weight $(m, n)$ Laplace operator given by \cite{ShaulGrossKohnen}, \cite{ShaulWeightChanging} which coincides with $4 \Omega_\kappa$ for $m = \kappa, n = 0$. In particular $\Omega_\kappa$ commutes with the weight $\kappa$ slash operator and satisfies
$$\Omega_\kappa q(Y)^s = s (s + \kappa - l / 2) q(Y)^s.$$
\end{rem}

\section{Siegel Theta Function}\label{sec:SiegelthetaFunction}

Let $p$ be a polynomial on $\IR^{(b^+, b^-)}$ which is homogeneous of degree $\kappa$ in the positive definite variables and independent negative definite variables. For an isometry $\nu : L \otimes \IR \to \IR^{(b^+, b^-)}$ we write $\nu^+$ and $\nu^-$ for the inverse image of $\IR^{(b^+, 0)}$ and $\IR^{(0, b^-)}$. For an element $\lambda \in L \otimes \IR$ we write $\lambda_{\nu^\pm}$ for the projection of $\lambda$ onto $\nu^\pm$. The positive definite \emph{majorant} associated to $\nu$ is then given by $q_\nu(\lambda) = q(\lambda_{\nu^+}) - q(\lambda_{\nu^-})$. For $\gamma \in L' / L, \tau \in \IH$ and an isometry $\nu : L \otimes \IR \to \IR^{(b^+, b^-)}$ we define the \emph{Siegel theta function}
\begin{align*}
\theta_\gamma(\tau, \alpha, \beta, \nu, p)
&:= \sum_{\lambda \in \gamma + L} \exp\left(\frac{\Delta}{8 \pi v}\right)(p)(\nu(\lambda + \beta)) \\
&\times e(\tau q((\lambda + \beta)_{\nu^+}) + \overline{\tau} q((\lambda + \beta)_{\nu^-} - (\lambda + \beta / 2, \alpha))),
\end{align*}
where $\Delta$ is the usual \emph{Laplace operator} on $\IR^{b^+ + b^-}$ and $\alpha, \beta \in L \otimes \IR$. Moreover we define
\begin{align*}
\Theta_L(\tau, \alpha, \beta, \nu, p)
&:= \sum_{\gamma \in L' / L} \theta_\gamma(\tau, \alpha, \beta, \nu, p) \frake_\gamma.
\end{align*}
For $\alpha = \beta = 0$ we write
\begin{align*}
\theta_\gamma(\tau, \nu, p) := \theta_\gamma(\tau, 0, 0, \nu, p)
&= \sum_{\lambda \in \gamma + L} \exp\left(-\frac{\Delta}{8 \pi v}\right)(p)(\nu(\lambda)) e(\tau q(\lambda_{\nu^+}) + \overline{\tau} q(\lambda_{\nu^-}))
\end{align*}
and
\begin{align*}
\Theta_L(\tau, \nu, p) := \Theta_L(\tau, 0, 0, \nu, p)
&= \sum_{\gamma \in L' / L} \theta_\gamma(\tau, \nu, p) \frake_\gamma.
\end{align*}
Using Poisson summation one obtains

\begin{thm}[{\cite[Theorem 4.1]{Borcherds}}]
For $(M, \phi) \in \Mp_2(\IZ), M = \left(\begin{smallmatrix}a & b \\ c & d\end{smallmatrix}\right)$ we have
$$\Theta_L(M\tau, a \alpha + b \beta, c \alpha + d \beta, \nu, p) = \phi(\tau)^{b^+ + 2 \kappa^+} \overline{\phi(\tau)}^{b^-} \rho_L(M, \phi) \Theta_L(\tau, \alpha, \beta, \nu, p).$$
In particular, for $\alpha = \beta = 0$, the theta function $\Theta_L(\tau, \nu, p)$ has weight $(\frac{b^+}{2} + \kappa^+, \frac{b^-}{2})$.
\end{thm}

Borcherds shows in \cite{Borcherds} that the theta function can be written as a Poincar\'e series. We will indicate the construction. Write $\Iso_0(L)$ for the primitive isotropic elements of $L$ and let $z \in \Iso_0(L), z' \in L'$ with $(z, z') = 1$. Let $N_z$ be the level of $z$ and define the lattice
$$K = L \cap z^\perp \cap z'^\perp.$$
Then $K$ has signature $(b^+ - 1, b^- - 1)$. For a vector $\lambda \in L \otimes \IR$ we write $\lambda_K$ for its orthogonal projection to $K \otimes \IR$, which is given by
$$\lambda_K = \lambda - (\lambda, z) z' + (\lambda, z)(z', z') z - (\lambda, z') z.$$
Let $\zeta \in L$ such that $(z, \zeta) = N_z$ and write
$$\zeta = \zeta_K + N_z z' + B z$$
for some $B \in \IQ$. Then we have
$$L = K \oplus \IZ \zeta + \IZ z.$$
Consider the sublattice
$$L_0' = \{ \lambda \in L' \vert (\lambda, z) \equiv 0 \bmod{N_z} \} \subseteq L'$$
and the projection
$$\pi : L_0' \to K', \lambda \mapsto \pi(\lambda) = \lambda_K + \frac{(\lambda, z)}{N_z} \zeta_K.$$
This projection induces a surjective map $L_0' / L \to K' / K$ which we also denote by $\pi$ and we have
$$L_0' / L = \{ \lambda \in L' / L \vert (\lambda, z) \equiv 0 \bmod{N_z} \}.$$
For an isometry $\nu : L \otimes \IR \to \IR^{(b^+, b^-)}$ we write $\omega^{\pm}$ for the orthogonal complement of $z_{\nu^\pm}$ in $\nu^\pm$. This yields a decomposition
$$L \otimes \IR = \omega^+ \oplus \IR z_{\nu^+} \oplus \omega^- \oplus \IR z_{\nu^-}$$
and for $\lambda \in L \otimes \IR$ we write $\lambda_{\omega^\pm}$ for the corresponding projections of $\lambda$ onto $\omega^\pm$. Additionally the map
$$\omega : L \otimes \IR \to \IR^{(b^+, b^-)}, \lambda \mapsto \nu(\lambda_{\omega^+} + \lambda_{\omega^-})$$
can be seen to be an isometry $K \otimes \IR \to \IR^{(b^+ - 1, b^- - 1)}$ by restriction. For a polynomial $p$ on $\IR^{(b^+, b^-)}$ as above we now define the homogeneous polynomials $p_{\omega, h}$ of degree $\kappa - h$ in the positive definite variables by
$$p(\nu(\lambda)) = \sum_{h} (\lambda, z_{\nu^+})^{h} p_{\omega, h}(\omega(\lambda)).$$
We have the following

\begin{thm}[{\cite[Theorem 5.2]{Borcherds}}]
Let $\mu = -z' + \frac{z_{\nu^+}}{2 z_{\nu^+}^2} + \frac{z_{\nu^-}}{2 z_{\nu^-}^2} \in L \otimes \IR$. Then
\begin{align*}
\theta_{\gamma + L}(\tau, \nu, p)
&= \frac{1}{\sqrt{2 v z_{\nu^+}^2}} \sum_{\substack{c, d \in \IZ \\ c \equiv (\gamma, z) \bmod N_z}} \sum_{h} (-2iv)^{-h} \\
&\times (c \overline{\tau} + d)^{h} e\left(-\frac{\lvert c \tau + d \rvert^2}{4 i v z_{\nu^+}^2} - (\gamma, z') d + q(z') cd \right) \\
&\times \theta_{K + \pi(\gamma - cz')}(\tau, d \mu_K, -c \mu_K, \omega, p_{\omega, h}).
\end{align*}
\end{thm}

We will need the following

\begin{lem}
Let $\gamma \in K' / K$. Then
\begin{align*}
&\rho_L(M) \sum_{m \in \IZ / N_z \IZ} \frake_{\gamma + \frac{mz}{N_z}}\left(-\frac{mn}{N_z}\right) \\
&= (\rho_K(M) \frake_\gamma) \sum_{m \in \IZ / N_z \IZ} \frake_{\frac{mz}{N_z} - nc z'}\left(-\frac{amn}{N_z} + q(z') acn^2 \right).
\end{align*}
\end{lem}

\begin{proof}
Write $\lambda = \gamma + \frac{mz}{N_z}$ We first consider the case $M = \left(\begin{smallmatrix} a & b \\ c & d \end{smallmatrix}\right)$ with $c = 0$. Then by Shintani's formula \ref{thm:Shintani}
\begin{align*}
\rho_L(M) \frake_{\lambda} &= \sqrt{i}^{(b^- - b^+)(1 - \sgn(d))} \sum_{\beta \in L' / L} \delta_{\beta, a \lambda} \frake_\beta(ab q(\beta)) \\
&= \sqrt{i}^{(b^- - b^+)(1 - \sgn(d))} \frake_{a \lambda}(ab q(a\lambda)) \\
&= \sqrt{i}^{((b^- - 1) - (b^+ - 1))(1 - \sgn(d))} \frake_{a \gamma}(ab q(a \gamma)) \frake_{\frac{amz}{N_z}} = (\rho_K(M) \frake_\gamma) \frake_{\frac{am z}{N_z}}.
\end{align*}
Multiplying by $e\left(-\frac{mn}{N_z}\right)$ and summing over $m \in \IZ / N_z \IZ$ yields the result. Let now $c \neq 0$. Again, Shintani's formula yields
$$\rho_L(M) \frake_\lambda = C_L(c) \sum_{\beta \in L' / L} \sum_{r \in L / cL} \frake_\beta\left(\frac{a(\beta + r, \beta + r) - 2(\lambda, \beta + r) + d(\lambda, \lambda)}{2c}\right),$$
where
$$C_L(c) = \frac{\sqrt{i}^{(b^- - b^+) \sgn(c)}}{\lvert c \rvert^{\frac{b^- + b^+}{2}} \sqrt{\lvert L' / L \rvert}} = \frac{\sqrt{i}^{((b^- - 1) - (b^+ - 1) \sgn(c)}}{\lvert c \rvert^{\frac{(b^- - 1) + (b^+ - 1)}{2}} \sqrt{\lvert K' / K \rvert}} \cdot \frac{1}{c N_z} = \frac{C_K(c)}{c N_z}.$$
Since $L = K \oplus \IZ \zeta \oplus \IZ z$ we can instead sum over $r + k \zeta + k' z, r \in K / cK, k, k' \in \IZ / c \IZ$ to obtain
\begin{align*}
\rho_L(M) \frake_\lambda 
&= C_L(c) \sum_{\beta \in L' / L} \sum_{\substack{r \in K / cK \\ k \in \IZ / c \IZ}} e\left(\frac{a(\beta + r, \beta + r) - 2(\gamma, k \zeta) + d (\gamma, \gamma)}{2c}\right) \\
&\times e\left(\frac{2a (\beta + r, k\zeta) + a(k \zeta, k' \zeta) - 2(\gamma, k \zeta) - 2(\frac{mz}{N_z}, \beta + k \zeta)}{2c}\right)\\
&\times \sum_{k' \in \IZ / c \IZ} e\left(\frac{2a (\beta, k' z) + 2a (k \zeta, k' \zeta)}{2c}\right) \frake_\beta \\
&= C_L(c) \sum_{\beta \in L' / L} \sum_{\substack{r \in K / cK \\ k \in \IZ / c \IZ}} e\left(\frac{a(\beta + r, \beta + r) - 2(\gamma, k \zeta) + d (\gamma, \gamma)}{2c}\right) \\
&\times e\left(\frac{2a (\beta + r, k\zeta) + a(k \zeta, k' \zeta) - 2(\gamma, k \zeta) - 2(\frac{mz}{N_z}, \beta + k \zeta)}{2c}\right)\\
&\times \sum_{k' \in \IZ / c \IZ} e \left(k' \frac{a (\beta + k \zeta, k' z)}{2c}\right)\frake_\beta.
\end{align*}
Now the last sum vanishes by orthogonality of characters unless
$$(\beta + k \zeta, z) = (\beta, z) + kN_z \equiv 0 \bmod{c},$$
in which case it sums to $c$. This yields
\begin{align*}
\rho_L(M) \frake_\lambda 
&= C_L(c) c \sum_{\beta \in L' / L} \sum_{\substack{r \in K / cK \\ k \in \IZ / c \IZ \\ c \mid kN_z + (\beta, z)}} e\left(\frac{a(\beta + r, \beta + r) - 2(\gamma, k \zeta) + d (\gamma, \gamma)}{2c}\right) \\
&\times e\left(\frac{2a (\beta + r, k\zeta) + a(k \zeta, k' \zeta) - 2(\gamma, k \zeta) - 2(\frac{mz}{N_z}, \beta + k \zeta)}{2c}\right) \frake_\beta.
\end{align*}
We now multiply this by $e\left(-\frac{mn}{N_z}\right)$ and sum over $m \in \IZ / N_z \IZ$ to obtain
\begin{align*}
& \rho_L(M) \sum_{m \in \IZ / N_z \IZ} \frake_{\gamma + \frac{mz}{N_z}}\left(-\frac{mn}{N_z}\right) \\
&= C_L(c) c \sum_{\beta \in L' / L} \sum_{\substack{r \in K / cK \\ k \in \IZ / c \IZ \\ c \mid kN_z + (\beta, z)}} e\left(\frac{a(\beta + r, \beta + r) - 2(\gamma, k \zeta) + d (\gamma, \gamma)}{2c}\right) \\
&\times e\left(\frac{2a (\beta + r, k\zeta) + a(k \zeta, k' \zeta) - 2(\gamma, k \zeta)}{2c}\right) \\
&\times \sum_{m \in \IZ / N_z \IZ} e\left(-\frac{(\frac{mz}{N_z}, \beta + k \zeta + nc z')}{c}\right) \frake_\beta.
\end{align*}
The latter sum again vanishes unless
$$(z, \beta + k \zeta + ncz') = (z, \beta + ncz') + kN_z \equiv 0 \bmod{cN_z},$$
in which case it is equal to $N_z$. In particular we have $(z, \beta + ncz') \equiv 0 \bmod{N_z}$ and thus $\beta + ncz' \in L_0' / L$. But every element in $L_0' / L$ can be written as $\alpha + \frac{mz}{N_z}$ for some $\alpha \in K' / K$ and $m \in \IZ / N_z \IZ$, which shows that $(\beta + ncz', \frac{z}{N_z}) = 0$. But this implies $k = 0$. Hence we obtain
\begin{align*}
& \rho_L(M) \sum_{m \in \IZ / N_z \IZ} \frake_{\gamma + \frac{mz}{N_z}}\left(-\frac{mn}{N_z}\right) \\
&= C_K(c) \sum_{\beta \in K' / K} \sum_{r \in K / cK} \frake_\beta\left(\frac{a(\beta + r, \beta + r) - 2(\gamma, \beta + r) + d(\gamma, \gamma)}{2c}\right) \\
&\times \sum_{m \in \IZ / N_z \IZ} \frake_{\frac{mz}{N_z} - ncz'}\left(-\frac{amn}{N_z} + q(z') acn^2\right) \\
&= (\rho_K(M) \frake_\gamma) \sum_{m \in \IZ / N_z \IZ} \frake_{\frac{mz}{N_z} - ncz'}\left(-\frac{amn}{N_z} + q(z') acn^2\right),
\end{align*}
which shows the claim.
\end{proof}

\begin{thm}
We have
\begin{align*}
\Theta_L(\tau, \nu, p)
&= \frac{1}{\sqrt{2 v z_{\nu^+}^2}} \Theta_{K}(\tau, \omega, p_{\omega, 0}) \sum_{m \in \IZ / N_z \IZ} \frake_{\frac{mz}{N_z}}\\
&+ \frac{1}{\sqrt{2 z_{\nu^+}^2}} \sum_{M \in \Gamma_\infty \bs \Gamma} \sum_{h} (-2i)^{-h} \sum_{n = 1}^\infty n^{h} \\
&\times \frac{(c \overline{\tau} + d)^{-\frac{b^-}{2} - \kappa^-} (c \tau + d)^{ -\frac{b^+}{2} - \kappa^+}}{\Im(M\tau)^{h + \frac{1}{2}}} \exp\left(-\frac{\pi n^2}{2\Im(M \tau) z_{\nu^+}^2}\right) \\
&\times \rho_L(M)^{-1} \left( \Theta_K(\tau, n \mu_K, 0, \omega, p_{\omega, h}) \sum_{m \in \IZ / N \IZ} \frake_{\frac{mz}{N}}\left(-\frac{mn}{N}\right)\right).
\end{align*}
\end{thm}

\begin{proof}
We have
\begin{align*}
\Theta_L(\tau, \nu, p)
&= \frac{1}{\sqrt{2 v z_{\nu^+}^2}} \sum_{c, d \in \IZ} \sum_{\substack{\gamma \in L' / L \\ c \equiv (\gamma, z) \bmod{N_z}}} \sum_{h} (-2iv)^{-h} \\
&\times (c \overline{\tau} + d)^{h} e\left(-\frac{\lvert c \tau + d \rvert^2}{4 i v z_{\nu^+}^2} - (\gamma, z') d + q(z') cd \right) \\
&\times \theta_{K + \pi(\gamma - cz')}(\tau, d \mu_K, -c \mu_K, \omega, p_{\omega, h}) \frake_\gamma.
\end{align*}
We make the change $\gamma \mapsto \gamma + cz'$ and sum over coprime $c, d$ to obtain
\begin{align*}
&\frac{1}{\sqrt{2 v z_{\nu^+}^2}} \sum_{\gamma \in L_0' / L} \theta_{K + \pi(\gamma)}(\tau, \omega, p_{\omega, 0}) \frake_\gamma \\
&+ \frac{1}{\sqrt{2 v z_{\nu^+}^2}} \sum_{(c, d) = 1} \sum_{h} (-2iv)^{-h} \sum_{n = 1}^\infty n^{h} (c \overline{\tau} + d)^{h} e\left(-\frac{n^2 \lvert c \tau + d \rvert^2}{4 i v z_{\nu^+}^2}\right) \\
&\times \sum_{\gamma \in L_0' / L} \theta_{K + \pi(\gamma)}(\tau, nd \mu_K, -nc \mu_K, \omega, p_{\omega, h}) \frake_{\gamma + ncz'}\left(-(\gamma, z') nd - q(z') n^2cd\right).
\end{align*}
The elements $\gamma \in L_0' / L$ are represented by $\gamma + \frac{mz}{N_z}$ for $\gamma \in K' / K, m \in \IZ / N_z \IZ$. Hence we obtain using the transformation formula for the theta function
\begin{align*}
&\sum_{\gamma \in L_0' / L} \theta_{K + \pi(\gamma)}(\tau, nd \mu_K, -nc \mu_K, \omega, p_{\omega, h}) \frake_{\gamma + ncz'}\left(-(\gamma, z') nd - q(z') n^2cd\right) \\
&= \sum_{\gamma \in K' / K} \theta_{K + \gamma}(\tau, nd \mu_K, -nc \mu_K, \omega, p_{\omega, h}) \frake_\gamma \sum_{m \in \IZ / N_z \IZ} \frake_{\frac{mz}{N_z} + ncz'}\left(-\frac{mnd}{N_z} - q(z') n^2cd\right) \\
&= \Theta_K(\tau, nd \mu_K, -nc \mu_K, \omega, p_{\omega, h}) \sum_{m \in \IZ / N_z \IZ} \frake_{\frac{mz}{N_z} + ncz'}\left(-\frac{mnd}{N_z} - q(z') n^2cd\right) \\
&= (c \tau + d)^{-\frac{b^+ - 1}{2} - \kappa + h} (c \overline{\tau} + d)^{-\frac{b^- - 1}{2}} \\
&\times \left(\rho_K^{-1}(M) \Theta_K(\tau, n \mu_K, 0, \omega, p_{\omega, h})\right) \sum_{m \in \IZ / N_z \IZ} \frake_{\frac{mz}{N_z} + ncz'}\left(-\frac{mnd}{N_z} - q(z') n^2cd\right) \\
&= (c \tau + d)^{-\frac{b^+ - 1}{2} - \kappa + h} (c \overline{\tau} + d)^{-\frac{b^- - 1}{2}} \\
&\times \rho_L(M)^{-1} \left( \Theta_K(\tau, n \mu_K, 0, \omega, p_{\omega, h}) \sum_{m \in \IZ / N_z \IZ} \frake_{\frac{mz}{N_z}}\left(-\frac{mn}{N_z}\right)\right),
\end{align*}
where $M = \left(\begin{smallmatrix}* & * \\ c & d\end{smallmatrix}\right) \in \SL_2(\IZ)$. This yields
\begin{align*}
\Theta_L(\tau, \nu, p)
&= \frac{1}{\sqrt{2 v z_{\nu^+}^2}} \Theta_{K}(\tau, \omega, p_{\omega, 0}) \sum_{m \in \IZ / N_z \IZ} \frake_{\frac{mz}{N_z}}\\
&+ \frac{1}{\sqrt{2 z_{\nu^+}^2}} \sum_{M \in \Gamma_\infty \bs \Gamma} \sum_{h} (-2i)^{-h} \sum_{n = 1}^\infty n^{h} \\
&\times \frac{(c \overline{\tau} + d)^{-\frac{b^-}{2}} (c \tau + d)^{ -\frac{b^+}{2} - \kappa}}{\Im(M\tau)^{h + \frac{1}{2}}} e\left(-\frac{n^2}{4 i \Im(M \tau) z_{\nu^+}^2}\right) \\
&\times \rho_L(M)^{-1} \left( \Theta_K(\tau, n \mu_K, 0, \omega, p_{\omega, h}) \sum_{m \in \IZ / N_z \IZ} \frake_{\frac{mz}{N_z}}\left(-\frac{mn}{N_z}\right)\right).
\end{align*}
\end{proof}

Let now $L$ be an even lattice of signature $(2, l)$ with $l \equiv 0 \bmod{2}$. Let $\kappa = \frac{l}{2} - 1 + k$ and $p(x_1, x_2) = (x_1 + ix_2)^\kappa$. Recall the identification $\nu : \calK^+ \to \Gr(L), [Z_L] = [X_L + iY_L] \mapsto \nu(Z_L) = \IR X_L + \IR Y_L$. For $Z_L = X_L + iY_L \in \tilde{\calK}^+$ the map
$$a X_L \to b Y_L \mapsto \lvert Y_L \rvert \begin{pmatrix}a \\ b\end{pmatrix}$$
defines an isometry $\nu_{Z_L} : \nu(Z_L) \to \IR^{(2, 0)}$ and by abuse of notation we will write $\nu_{Z_L}$ for every isometry which equals $\nu_{Z_L}$ on $\nu(Z_L)$. For $\lambda \in V(\IR)$ we write $\lambda_{Z_L}$ and $\lambda_{Z_L^\perp}$ for the projection of $\lambda$ to $\nu(Z_L)$ and $\nu(Z_L)^\perp$. Then $q(\lambda) = q(\lambda_{Z_L}) + q(\lambda_{Z_L^\perp})$ and we denote by $q_{Z_L}(\lambda) = q(\lambda_{Z_L}) - q(\lambda_{Z_L^\perp})$ the positive definite majorant. Now, the map
$$\lambda \mapsto p(\nu_{Z_L}(\lambda)) = \frac{(\lambda, Z_L)^\kappa}{\lvert Y_L \rvert^\kappa}$$
is well-defined, since $p$ only depends on the positive definite variables. Thus we define
\begin{align*}
\Theta_L(\tau, Z)
&:= \frac{i^\kappa v^{\frac{l}{2}}}{2 \lvert Y_L \rvert^\kappa} \Theta_L(\tau, \nu_{Z_L}, p) \\
&= \frac{v^{\frac{l}{2}}}{2 (-2i)^{\kappa}} \sum_{\lambda \in L'} \frac{(\lambda, Z_L)^\kappa}{q(Y_L)^\kappa} \frake_\lambda(\tau q(\lambda_{Z_L}) + \overline{\tau} q(\lambda_{Z_L^\perp}))
\end{align*}
which is modular of weight $k = 1 - \frac{l}{2} + \kappa$ in $\tau = u + i v \in \IH$ and weight $\kappa$ in $Z = X + iY \in \IH_l$. Let $z \in \Iso_0(L)$ of level $N_z$ and $z' \in L'$ with $(z, z') = 1$ write $K = L \cap z^\perp \cap z'^\perp$. Further, let $d \in \Iso_0(K)$ of level $N_d$ and $d' \in K'$ with $(d, d') = 1$ and $D = K \cap d^\perp \cap d'^\perp$. Moreover we let $\tilde{z} = z' - q(z') z, \tilde{d} = d' - q(d') d$. We expand $\Theta_L(\tau, Z)$ with respect to $z, z', K$ to obtain
\begin{align*}
\Theta_L(\tau, Z)
&= \frac{v^{\frac{l - 1}{2}}}{2 \sqrt{2} i^\kappa \lvert Y \rvert^{\kappa - 1}} \Theta_K(\tau, Y / \lvert Y \rvert, p_{Y, 0}) \sum_{m_z \in \IZ / N_z \IZ} \frake_{\frac{m_zz}{N_z}} \\
&+ \frac{i^\kappa}{2\sqrt{2} \lvert Y \rvert^{\kappa - 1}} \sum_{M \in \Gamma_\infty \bs \Gamma} \sum_{h = 0}^\kappa \sum_{n = 1}^\infty n^h \frac{(c \tau + d)^{-k}}{(-2i)^h \Im(M\tau)^{h + \frac{1}{2} - \frac{l}{2}}} \exp\left(-\frac{q(Y) n^2 \pi}{\Im(M\tau)}\right) \\
&\times \rho_L^{-1}(M) \left(\Theta_K(M \tau, nX, 0, Y / \lvert Y \rvert, p_{Y, h}) \sum_{m_z \in \IZ / N_z Z} \frake_{\frac{m_zz}{N_z}}\left(-\frac{m_zn}{N_z} \right)\right).
\end{align*}
and we expand $\Theta_K(\tau, Y / \lvert Y \rvert, p_{Y, 0})$ in the first summand with respect to $d, d', D$ to obtain
\begin{align*}
\Theta_L(\tau, Z)
&= \delta_{\kappa, 0} \frac{v^{\frac{l}{2} - 1} q(Y)}{2 y_1} \Theta_D(\tau) \sum_{\substack{m_d \in \IZ / N_d \IZ \\ m_z \in \IZ / N_z \IZ}} \frake_{\frac{m_d d}{N_d} + \frac{m_z z}{N_z}} \\
&+ \frac{q(Y)}{\lvert Y \rvert^{\kappa} 2^{\kappa + 1} y_1} \sum_{M \in \Gamma_\infty \bs \Gamma} \sum_{n = 1}^\infty n^\kappa \frac{(c \tau + d)^{-k}}{\Im(M\tau)^{k}} \exp\left(-\frac{q(Y) \pi n^2}{\Im(M \tau) y_1^2}\right) \\
&\times \rho_L(M)^{-1} \left(\Theta_D(M \tau, n Y_D / y_1, 0) \sum_{\substack{m_d \in \IZ / N_d \IZ \\ m_z \in \IZ / N_z \IZ}} \frake_{\frac{m_d d}{N_d} + \frac{m_z z}{N_z}}\left(-\frac{m_d n}{N_d}\right)\right) \\
&+ \frac{i^\kappa}{2\sqrt{2} \lvert Y \rvert^{\kappa - 1}} \sum_{M \in \Gamma_\infty \bs \Gamma} \sum_{h = 0}^\kappa \sum_{n = 1}^\infty n^h \frac{(c \tau + d)^{-k}}{(-2i)^h \Im(M\tau)^{h + \frac{1}{2} - \frac{l}{2}}} \exp\left(-\frac{q(Y) n^2 \pi}{\Im(M\tau)}\right) \\
&\times \rho_L^{-1}(M) \left(\Theta_K(M \tau, nX, 0, Y / \lvert Y \rvert, p_{Y, h}) \sum_{m_z \in \IZ / N_z Z} \frake_{\frac{m_zz}{N_z}}\left(-\frac{m_zn}{N_z} \right)\right).
\end{align*}
Observe that the second and third summand can be rewritten using the weight $k$ slash operator $\vert_{k, L}$. We want to get bounds for $\Theta_L(\tau, Z)$ and $\Omega_k \overline{\Theta_L(\tau, Z)}$. According to \cite[Proposition 2.5]{ShaulGrossKohnen} we have
$$\Omega_\kappa \overline{\Theta_L(\tau, Z)} = \overline{\Delta_k \Theta_L(\tau, Z)},$$
where $\Delta_k = -v^2 \left(\frac{\partial^2}{\partial u^2} + \frac{\partial^2}{\partial v^2}\right) + i k v \left(\frac{\partial}{\partial x} + i \frac{\partial}{\partial y}\right)$ is the weight $k$ Laplace operator in $\tau$. Let
$$f(\tau, Z, n) := \frac{\sqrt{q(Y) \pi} n}{v^\kappa y_1} \exp\left(- \frac{q(Y) \pi n^2}{v y_1^2}\right) \Theta_D(\tau, n Y_D /(\lvert Y \rvert y_1), 0)$$
and
$$g(\tau, Z, n) := \lvert Y \rvert^{1 - \kappa} v^{-h} \exp\left(- \frac{q(Y) n^2 \pi}{v}\right) \Theta_K(\tau, nX, 0, Y / \lvert Y \rvert, p_{Y, h})$$
so that $\Theta_L(\tau, Z)$ is given by
\begin{align*}
&\delta_{\kappa, 0} \frac{q(Y)}{2 y_1} \Theta_D(\tau) \sum_{\substack{m_d \in \IZ / N_d \IZ \\ m_z \in \IZ / N_z \IZ}} \frake_{\frac{m_d d}{N_d} + \frac{m_z z}{N_z}} \\
&+ \frac{\lvert Y \rvert^{1 - \kappa}}{2^{\kappa + 1} \sqrt{2}} \sum_{M \in \Gamma_\infty \bs \Gamma} \sum_{n = 1}^\infty n^{\kappa - 1} \left(f(\tau, Z, n) \sum_{\substack{m_d \in \IZ / N_d \IZ \\ m_z \in \IZ / N_z \IZ}} \frake_{\frac{m_d d}{N_d} + \frac{m_z z}{N_z}}\left(-\frac{m_d n}{N_d}\right)\right) \bigg\vert_{k, L} M \\
&+ \frac{i^\kappa}{2 \sqrt{2}} \sum_{M \in \Gamma_\infty \bs \Gamma} \sum_{h = 0}^\kappa \sum_{n = 1}^\infty \frac{n^h}{(-2i)^h} \left(g(\tau, Z, n) \sum_{m_z \in \IZ / N_z Z} \frake_{\frac{m_zz}{N_z}}\left(-\frac{m_zn}{N_z} \right)\right) \bigg\vert_{k, L} M.
\end{align*}
Since the Laplace operator commutes with the slash operator, we only have to find bounds for $f, \Delta_k f$ and $g, \Delta_k g$. We will need the following elementary lemma.

\begin{lem}\label{lem:ElementaryLemma}
Let $a > 0, b, c \geq 0, n \in \IN$.
\begin{enumerate}
\item[(i)] The function $x^n \exp(- ax^2)$ has a maximum given by $\left(\frac{n}{2 a e}\right)^{\frac{n}{2}}$.
\item[(ii)] The function $x^n \exp(- ax)$ has a maximum on $x \geq 0$ given by $\left(\frac{n}{ae}\right)^n$.
\item[(iii)] The function $bx + \frac{c}{x}$ has a minimum on $x > 0$ given by $2 \sqrt{bc}$
\end{enumerate}
\end{lem}

We start with a bound for
$$f(\tau, Z, n) = \frac{\sqrt{q(Y)} n}{v^\kappa y_1} \exp\left(- \frac{q(Y) \pi n^2}{v y_1^2}\right) \Theta_D(\tau, n Y_D /(\lvert Y \rvert y_1), 0).$$
Observe that the absolute value of $f$ and $\Delta_k f$ can be bounded by finite sums of the form
$$v^{-i} (y n)^j \exp\left(- \frac{(y n)^2}{v}\right) \sum_{\lambda \in D'} q(\lambda)^l \exp(- 2 \pi v q(\lambda)).$$
We have

\begin{lem}
Let
$$f_{i, j, l}(v, y, n) := v^{-i} (y n)^j \exp\left(- \frac{3 (y n)^2}{v}\right) \sum_{\lambda \in D'} q(\lambda)^l \exp(- 2 \pi v q(\lambda)).$$
Assume that $v \in \IR_{>0}, y > C > 0$ and $n \geq 1$. Then we have the bound
$$f_{i, j, l}(v, y, n) \ll v^{-i+j} \exp\left(- \frac{C^2 n^2}{3v}\right),$$
where the implied constant is independent of $v, y, n$.
In particular, we obtain with $y^2 = \frac{q(Y) \pi}{y_1^2}, C^2 = \frac{\pi}{t^2}$
$$f(\tau, Z, n) \ll v^{s_1} \exp\left(- \frac{\pi n^2}{3v t^2}\right)$$
and
$$\Delta_k f(\tau, Z, n) \ll v^{s_2} \exp\left(- \frac{\pi n^2}{3v t^2}\right)$$
for some $s_1, s_2 \in \IR$.
\end{lem}

\begin{proof}
We split up the exponential term and use Lemma \ref{lem:ElementaryLemma} to obtain
\begin{align*}
(yn)^j \exp\left(-\frac{(yn)^2}{v}\right)
&= (yn)^j \exp\left(-\frac{(yn)^2}{3v}\right) \exp\left(-\frac{(yn)^2}{3v}\right) \exp\left(-\frac{(yn)^2}{3v}\right) \\
&\ll v^j \exp\left(-\frac{(yn)^2}{3v}\right) \exp\left(-\frac{C^2}{3v}\right).
\end{align*}
This yields
\begin{align*}
& f_{i, j, l}(v, y, n) \\
&\ll v^{-i + j} \exp\left(- \frac{(y n)^2}{3v}\right) \sum_{\lambda \in D'} q(\lambda)^l \exp\left(- 2 \pi v q(\lambda) - \frac{C^2}{3v}\right) \\
&\ll v^{-i+j} \exp\left(- \frac{(y n)^2}{3v}\right) \sum_{\lambda \in D'} q(\lambda)^{l} \exp\left(- 2 C \sqrt{2 \pi q(\lambda)}\right) \ll \exp\left(- \frac{(yn)^2}{3v}\right)
\end{align*}
where we have used Lemma \ref{lem:ElementaryLemma} again.
\end{proof}

Next consider
$$g_h(\tau, Z, n) = \lvert Y \rvert^{1 - \kappa} v^{-h} \exp\left(- \frac{q(Y) n^2 \pi}{v}\right) \Theta_K(\tau, nX, 0, Y / \lvert Y \rvert, p_{Y, h}).$$
The theta function is given by $v^{\frac{l - 1}{2}}$ times
\begin{align*}
& \sum_{\lambda \in K'} \exp\left(-\frac{\Delta}{8 \pi v}\right)(p_{Y, h})(\omega_Y(\lambda)) \frake_\lambda\left(u q(\lambda) + iv\left(\frac{(\lambda, Y)^2}{Y^2} - q(\lambda)\right) - (\lambda, nX)\right) \\
&= \sum_{\lambda \in K'} \exp\left(-\frac{\Delta}{8 \pi v}\right)(p_{Y, h})(\omega_Y(\lambda)) \frake_\lambda\left(u q(\lambda) - (\lambda, nX)\right) \exp\left(- 2 \pi v\left(\frac{(\lambda, Y)^2}{Y^2} - q(\lambda)\right)\right).
\end{align*}
The terms
$$\exp\left(- \frac{\Delta}{8 \pi v}\right)(p_{Y, h})(\omega_Y(\lambda))$$
are given by a finite sum of constants times terms of the form
$$v^{-j} \lvert Y \rvert^h (\lambda, Y / \lvert Y \rvert)^{\kappa - h - 2j}.$$
Again, we see that the absolute value of $g_h$ and $\Delta_k g_h$ can be bounded by a finite sum of terms of the form
$$\lvert Y \rvert^{h} v^{- i} \exp\left(-\frac{q(Y) n^2 \pi}{v}\right) \sum_{\lambda \in K'} \left(\frac{(\lambda, Y)^2}{Y^2}\right)^j q(\lambda)^l \exp\left(- 2 \pi v\left(\frac{(\lambda, Y)^2}{Y^2} - q(\lambda)\right)\right).$$
\begin{lem}
Let $g_{h, i, j, l}(v, Y, n)$ denote the function
\[\lvert Y \rvert^h v^{- i} \exp\left(-\frac{q(Y) n^2 \pi}{v}\right) \sum_{\lambda \in K'} \left(\frac{(\lambda, Y)^2}{Y^2}\right)^j q(\lambda)^l \exp\left(- 2 \pi v\left(\frac{(\lambda, Y)^2}{Y^2} - q(\lambda)\right)\right)\]
for $v \in \IR_{>0}, Y \in \calR_t$ and $n \geq 1$. Then
$$g_{h, i, j, l}(v, Y, n) \ll v^{\frac{h}{2} - i - j} \exp\left(- \frac{q(Y) n^2 \pi}{4v}\right),$$
where the implied constant is independent of $v, Y, n$. In particular, we have
$$g_h(\tau, Z, n) \ll v^{s_1}\exp\left(- \frac{q(Y) n^2 \pi}{4v}\right)$$
and
$$\Delta_k g_h(\tau, Z, n) \ll v^{s_2} \exp\left(- \frac{q(Y) n^2 \pi}{4v}\right)$$
for some $s_1, s_2 \in \IR$ independent of $v, Y, n, h$.
\end{lem}

\begin{proof}
We have
$$\left(\frac{(\lambda, Y)^2}{Y^2}\right)^{j} \leq (y_2 / y_1 \lambda_1^2 + y_1 / y_2 \lambda_2^2 + q(\lambda_D))^{j}$$
and by Lemma \ref{lem:SiegelDomainInequality} the exponential term can be bounded by
\begin{align*}
\exp\left(- 2 \pi \varepsilon v\left(y_2 / y_1 \lambda_1^2 + y_1 / y_2 \lambda_2^2 + q(\lambda_D) \right)\right)
\end{align*}
on $\calR_t$, which yields the bound
\begin{align*}
& g_{h, i, j, l}(v, Y, n) \\
&\leq \lvert Y \rvert^h v^{- i} \exp\left(-\frac{q(Y) n^2 \pi}{v}\right) \sum_{\lambda \in K'} q(\lambda)^l \exp\left(- \pi \varepsilon v\left(y_2 / y_1 \lambda_1^2 + y_1 / y_2 \lambda_2^2 + q(\lambda_D) \right) \right) \\
&\times (y_2 / y_1 \lambda_1^2 + y_1 / y_2 \lambda_2^2 + q(\lambda_D))^{j} \exp\left(- \pi \varepsilon v\left(y_2 / y_1 \lambda_1^2 + y_1 / y_2 \lambda_2^2 + q(\lambda_D) \right)\right).
\end{align*}
Again we split up the exponential term Lemma \ref{lem:ElementaryLemma} to obtain
\begin{align*}
\lvert Y \rvert^h \exp\left(-\frac{q(Y) n^2 \pi}{v}\right) &= \lvert Y \rvert^h \exp\left(-\frac{q(Y) n^2 \pi}{4v}\right) \exp\left(-\frac{3q(Y) n^2 \pi}{4v}\right) \\
&\ll v^{\frac{h}{2}} \exp\left(-\frac{3q(Y) n^2 \pi}{4v}\right).
\end{align*}
We obtain
\begin{align*}
& g_{h, i, j, l}(v, Y, n) \\
&\ll v^{- i - j + \frac{h}{2}} \exp\left(-\frac{3q(Y) n^2 \pi}{4v}\right) \sum_{\lambda \in K'} q(\lambda)^l \exp\left(- \pi \varepsilon v\left(y_2 / y_1 \lambda_1^2 + y_1 / y_2 \lambda_2^2 + q(\lambda_D) \right) \right),
\end{align*}
where we have used Lemma \ref{lem:ElementaryLemma} again. Now use $y_1 > t^{-1}, y_2 / y_1 > t^{-2}, q(Y) > \frac{y_1 y_2}{1 + t^4}$ and $q(Y) > t^{-4}$ on every Siegel domain $\calS_t$ to obtain
\begin{align*}
& g_{h, i, j, l}(v, Y, n) \\
&\ll v^{- i - j + \frac{h}{2}} \exp\left(-\frac{q(Y) n^2 \pi}{4 v}\right) \\
&\sum_{\lambda \in K'} q(\lambda)^l \exp\left(- \pi \varepsilon v\left(\frac{y_2 \lambda_1^2}{y_1} + q(\lambda_D) \right) - \frac{q(Y) n^2 \pi}{4 v} - 2 \pi \varepsilon v \frac{y_1 \lambda_2^2}{y_2} - \frac{q(Y) n^2 \pi}{4 v} \right) \\
&\leq v^{- i - j + \frac{h}{2}} \exp\left(-\frac{q(Y) n^2 \pi}{4 v}\right) \\
&\sum_{\lambda \in K'} q(\lambda)^l \exp\left(- \pi \varepsilon v\left(\frac{\lambda_1^2}{t^2} + q(\lambda_D) \right) - \frac{\pi}{4 t^4 v} - 2 \pi \varepsilon v \frac{y_1 \lambda_2^2}{y_2} - \frac{y_1 y_2 \pi}{4 v (1 + t^4)} \right) \\
&\ll v^{- i - j + \frac{h}{2}} \exp\left(-\frac{q(Y) n^2 \pi}{4 v}\right) \sum_{\lambda \in K'} q(\lambda)^l \exp\left(- \pi \sqrt{\varepsilon} \left(\lvert \lambda_1 \rvert + t \lvert \lambda_D \rvert + \frac{\sqrt{2} \lvert \lambda_2 \rvert}{t \sqrt{1 + t^4}}\right) \right) \\
&\ll v^{- i - j + \frac{h}{2}} \exp\left(-\frac{q(Y) n^2 \pi}{4 v}\right).
\end{align*}
\end{proof}

\begin{prop}\label{prop:ThetaSiegelDomainGrowth}
Let $C \in \IR_{>0}$. The theta function $\Theta_L(\tau, Z)$ is bounded by a constant times
$$v^s \lvert Y \rvert^{1 - \kappa} \left(1 + \delta_{\kappa, 0} \frac{\lvert Y \rvert}{y_1}\right)$$
for all $Z \in \calS_t$ and $\tau \in \IH$ with $\Im(\tau) > C$ and some $s \in \IR$. The constant only depends on the constant $C$. Similarly, the function $\Omega_\kappa \overline{\Theta_L} = \overline{\Delta_k \Theta_L}$ is bounded by
$$q(Y)^{\frac{1 - \kappa}{2}}.$$
\end{prop}

\begin{proof}
This is now a direct consequence of the previous two Lemmas together with the fact that the Laplace operator $\Delta_k$ commutes with the slash operator $\vert_{k, L}$.
\end{proof}

\begin{cor}\label{cor:ThetaSquareIntegrable}
Both, the theta function $\Theta_L$ and $\Omega_\kappa \overline{\Theta_L}$ are square-integrable for $l \geq 3$ and $\kappa = \frac{l}{2} - 1 + k > 0$.
\end{cor}

\begin{proof}
The square is bounded by
$$q(Y)^{1 - \kappa}$$
and hence we have to show that
$$\int_{\calS_t} q(Y) \frac{\mathrm{d}X \mathrm{d}Y}{q(Y)^l} < \infty.$$
By Lemma \ref{lem:SiegelDomainIntegrable} this is the case if $l \geq 3$.
\end{proof}

\section{Orthogonal Eisenstein Series}

Let $L$ be an even lattice of signature $(2,l)$ and let $\kappa \in \IZ$. Let $\lambda \in \Iso_0(L)$ and recall the tube domain representation $\IH_l$ corresponding to a fixed $0$-dimensional cusp $z$. Denote by $\Gamma(L)_\lambda \subseteq \Gamma(L)$ the stabilizer of $\lambda$ in $\Gamma(L)$ and write $\sigma_{\lambda} \in O^+(V)$ for an element satisfying $\sigma_{\lambda} \lambda = z$. Then $q(Y)^s \vert_\kappa \sigma_{\lambda}$ has weight $\kappa$ with respect to $\Gamma(L)_\lambda$. Hence we define the non-holomorphic Eisenstein series
\begin{align*}
\calE_{\kappa, \lambda}(Z, s)
&:= \sum_{\sigma \in \Gamma(L)_\lambda \bs \Gamma(L)} q(Y)^s \vert_\kappa \sigma_{\lambda}\sigma \\
&= \sum_{\sigma \in \Gamma(L)_\lambda \bs \Gamma(L)} j(\sigma_{\lambda} \sigma, Z)^{-\kappa} \left(\frac{q(Y)}{\lvert j(\sigma_{\lambda} \sigma, Z) \rvert^2}\right)^s \\
&= \sum_{\sigma \in \Gamma(L)_\lambda \bs \Gamma(L)} (\lambda, \sigma(Z_L))^{-\kappa} \left(\frac{q(Y)}{\lvert (\lambda, \sigma(Z_L)) \rvert^2}\right)^s
\end{align*}
for $Z \in \IH_l$ and $\Re(s) \gg 0$. The Eisenstein series does not depend on the choice of $\sigma_\lambda$. We have $\Gamma(L)_{-\lambda} = \Gamma(L)_\lambda$ and $\calE_{\kappa, -\lambda}(Z, s) = (-1)^\kappa \calE_{\kappa, \lambda}$. We have
$$\Omega_\kappa \calE_{\kappa, \lambda}(Z, s) = s\left(s + \kappa - \frac{l}{2}\right) \calE_{\kappa, \lambda}(Z, s),$$
i.e. the harmonic points of $\calE_{\kappa, \lambda}$ are $s = 0$ and $s = \frac{l}{2} - \kappa$. Recall the map
$$\pi_L : \Gamma(L) \bs \Iso_0(L) \to \Gamma(L) \bs \Iso_0(L') \to \Iso(L' / L).$$
For $\delta \in \Iso(L' / L)$ we let
$$\calG_{\kappa, \delta} := \sum_{\lambda \in \pi_L^{-1}(\delta)} \calE_{\kappa, \lambda},$$
in particular, $\calG_{\kappa, \delta} = 0$ if the preimage is empty and if $\delta = -\delta$ for odd $\kappa$.

For a vector-valued modular form $f : \IH \to \IC[L ' / L]$ of weight $k$ consider the additive Borcherds lift
$$\Phi(Z, f) := \int_{\SL_2(\IZ) \bs \IH}^{\text{reg}} \langle f(\tau), \Theta_L(\tau, Z) \rangle \mathrm{d}\tau,$$
if it exists. Here, the regularization is defined by the constant term at $t = 0$ of the Laurent expansion of
$$\lim_{T \to \infty} \int_{\calF_T} \langle f(\tau), \Theta_L(\tau, Z) \rangle v^{-t} \mathrm{d}\tau,$$
where $\calF_T$ is a truncated fundamental domain. For the lift of the vector-valued non-holomorphic Eisenstein series $E_{k, \beta}(\tau, s)$ we write $\Phi_{k, \beta}(Z, s)$. We have the following

\begin{thm}[{\cite[Theorem 8.1]{Kiefer}}]\label{thm:ThetaLiftIsEisensteinSeries}
The theta lift is equal to
\begin{align*}
\Phi_{k, \beta}(Z, s) &= \frac{\Gamma(s + \kappa)}{(-2\pi i)^{\kappa}\pi^{s}} \sum_{\lambda \in \Gamma(L) \bs \Iso_0(L)} N_\lambda^{2s + \kappa} \zeta_+^{k_{\lambda\beta}}(2s + \kappa) \calE_{\kappa, \lambda}(Z, s) \\
&= \frac{\Gamma(s + \kappa)}{(-2\pi i)^{\kappa}\pi^{s}} \sum_{\delta \in \Iso(L' / L)} N_\delta^{2s + \kappa} \zeta_+^{k_{\delta\beta}}(2s + \kappa) \calG_{\kappa, \delta}(Z, s),
\end{align*}
where $k_{\delta\beta} \in \IZ / N_\delta \IZ$ with $\beta = k_{\delta \beta} \delta$ (and the corresponding summands vanish if such a $k_{\delta \beta}$ does not exist), $N_\lambda$ is the level of $\lambda$ and $N_\delta$ is the order of $\delta$.
\end{thm}

We also have the following

\begin{thm}[{\cite[Theorem 8.2]{Kiefer}}]
The theta lifts $\Phi_{k,\beta}(Z, s)$ generate the space of Eisenstein series $\calG_{\kappa, \delta}(Z, s)$ for $\Gamma(L)$. In particular, if $\pi_L$ is injective (resp. surjective), then the theta lift is surjective (resp. injective) onto (resp. on) non-holomorphic Eisenstein series.
\end{thm}

Hence, instead of evaluating the Eisenstein series $\calE_{\kappa, \lambda}$ at $s = 0$, we will instead consider the theta lifts $\Phi_{k, \beta}$. Their Fourier expansion is given by

\begin{thm}[{\cite[Theorem 8.6]{Kiefer}}]\label{thm:ThetaLiftFourierExpansion}
Let $z \in \Iso_0(L)$ of level $N_z$ and let $z' \in L'$ with $(z, z') = 1$. The theta lift $\Phi_{k,\beta}(Z, s)$ has the Fourier expansion in the $0$-dimensional cusp $z$ given by
\begin{align*}
&\frac{i^\kappa }{2 \sqrt{2}\lvert Y \rvert^{\kappa-1}} \Phi^K_{k,\beta}\left(\frac{Y}{\lvert Y \rvert}, s\right) + \sum_{\lambda \in K'} b_{k,\beta}(\lambda, Y, s) e(\lambda, X),
\end{align*}
where $\Phi^K_{k,\beta}\left(\frac{Y}{\lvert Y \rvert}, s\right)$ is a theta lift corresponding to the sublattice $K$. The Fourier coefficient $b_{k,\beta}(0, Y, s)$ is given by
\begin{align*}
& \sum_{b \in \IZ / N_z \IZ} \bigg(\frac{\Gamma(s + \kappa) N_z^{2s + \kappa}}{(-2 \pi i)^\kappa \pi^s} q(Y)^s (\delta_{\beta, \frac{bz}{N_z}} + (-1)^\kappa \delta_{- \beta, \frac{bz}{N_z}}) \zeta_+^b(2s + \kappa) \\
&+ \frac{\Gamma(1 - s - k + \kappa) N_z^{2 - 2s - 2k + \kappa}}{(-2 \pi i)^\kappa \pi^{1 - s - k}} q(Y)^{1 - s - k} c_{k,\beta}\left(\frac{bz}{N_z}, 0, s\right) \zeta_+^b(2 - 2s - 2k + \kappa)\bigg).
\end{align*}
For $q(\lambda) = 0, \lambda \neq 0$ the coefficient $b_{k,\beta}(\lambda, Y, s)$ is given by
\begin{align*}
&\frac{2 \lvert(\lambda, Y) \rvert^{\frac{1}{2}}}{2^\kappa} \sum_{b \in \IZ / N_z \IZ} \bigg(\frac{q(Y)^s}{\lvert (\lambda, Y) \rvert^{s}} e\left(- \frac{(\lambda, \zeta)}{N_z}\right) \\
&\times \sum_{n \mid \lambda} n^{2s - 1 + \kappa} (\delta_{\beta, \frac{\lambda}{n} - \frac{(\lambda, \zeta)}{nN_z} z + \frac{bz}{N_z}} + (-1)^\kappa \delta_{-\beta, \frac{\lambda}{n} - \frac{(\lambda, \zeta)}{nN_z} z + \frac{bz}{N_z}}) e\left(\frac{nb}{N_z}\right) \\
&\times \sum_{h = 0}^\infty \sum_{j = 0}^\infty \frac{(-1)^{j}}{(4 \pi \lvert(\lambda, Y)\rvert)^j j!} \binom{\kappa}{h} \frac{(\kappa - h)!}{(\kappa-h-2j)!} \\
&\times \left(\frac{(\lambda, Y)}{\lvert (\lambda, Y) \rvert}\right)^{\kappa-h} K_{s - \frac{1}{2} + \kappa - h - j}(2 \pi \lvert (\lambda, Y) \rvert)\\
&+ \frac{q(Y)^{1 - s - k}}{\lvert (\lambda, Y) \rvert^{1 - s - k}} e\left(- \frac{(\lambda, \zeta)}{N_z}\right) \\
&\times \sum_{n \mid \lambda} n^{1 - 2s - 2k + \kappa} c_{k,\beta}\left(\frac{\lambda}{n} - \frac{(\lambda, \zeta)}{nN_z} z + \frac{bz}{N_z}, 0, s\right) e\left(\frac{nb}{N_z}\right) \\
&\times \sum_{h = 0}^\infty \sum_{j = 0}^\infty \frac{(-1)^{j}}{(4 \pi \lvert(\lambda, Y)\rvert)^j j!} \binom{\kappa}{h} \frac{(\kappa - h)!}{(\kappa-h-2j)!} \\
&\times \left(\frac{(\lambda, Y)}{\lvert (\lambda, Y) \rvert}\right)^{\kappa-h} K_{\frac{1}{2} - s - k + \kappa - h - j}(2 \pi \lvert (\lambda, Y ) \rvert)\bigg).
\end{align*}
For $q(\lambda) \neq 0$ the coefficient $b_{k,\beta}(\lambda, Y, s)$ is given by
\begin{align*}
& \frac{1}{\sqrt{2}\lvert Y \rvert^{\kappa-1}} \sum_{b \in \IZ / N_z \IZ} \sum_{h = 0}^\infty (2i)^{-h} \sum_{j = 0}^\infty \frac{(-1)^j i^{h}}{(8 \pi)^j j!} \binom{\kappa}{h} \frac{(\kappa-h)!}{(\kappa-h-2j)!} \lvert Y \rvert^h \left(\frac{(\lambda, Y)}{\lvert Y \rvert}\right)^{\kappa-h - 2j} \\
&\times e\left(- \frac{(\lambda, \zeta)}{N_z}\right) \sum_{n \mid \lambda} n^{2j - \kappa + 2h} e\left(\frac{nb}{N_z}\right) c_{k,\beta}\left(\frac{\lambda}{n} - \frac{(\lambda, \zeta)}{nN_z} z + \frac{bz}{N_z}, \frac{q(\lambda)}{n^2}, s\right) \\
&\times \int_0^\infty \exp\left(-\frac{\pi n^2}{2vz_Z^2} - \frac{2 \pi v q_\omega(\lambda)}{n^2}\right) \calW_s\left(4 \pi \frac{q(\lambda)}{n^2}v\right) v^{-\frac{3}{2} + \kappa - h - j} \mathrm{d}v.
\end{align*}
\end{thm}

\section{Theta Lifts at Harmonic Points}

We will now consider the theta lifts at their harmonic points $s = 0$ and for $k = 0$ at $s = 1$. We set $\Phi_\frakv(Z) = \Phi_\frakv(Z, 0)$ and $b_{k,\beta}(\lambda, Y) = b_{k,\beta}(\lambda, Y, 0)$. We need the following

\begin{lem}\label{lem:SumOfKBesselFunctions}
We have
\begin{align*}
&\sum_{h = 0}^\infty \sum_{j = 0}^\infty \frac{(-1)^{j}}{(4 \pi \lvert y \rvert)^j j!} \binom{\kappa}{h} \frac{(\kappa-h)!}{(\kappa-h-2j)!} \left(\frac{y}{\lvert y \rvert}\right)^{\kappa-h} K_{-\frac{1}{2} + \kappa - h - j}\left(2 \pi \lvert y \rvert \right) \\
&= 2^{\kappa - 1} y^{-\frac{1}{2}} e^{-2 \pi y}.
\end{align*}
if $y > 0$ and the left hand side vanishes otherwise.
\end{lem}

\begin{proof}
This is done in \cite[Theorem 14.3]{Borcherds}. We use the formula
\begin{align*}
K_{n + \frac{1}{2}}(z) = \sqrt{\frac{\pi}{2z}} e^{-z} \sum_{0 \leq m} (2z)^{-m} \frac{(n + m)!}{m! (n - m)!}
\end{align*}
and observe that the terms with $m \neq 0, j \neq 0$ cancel out according to \cite[Corollary 14.2]{Borcherds}. Hence we obtain
\begin{align*}
&\sum_{h = 0}^\infty \sum_{j = 0}^\infty \frac{(-1)^{j}}{(4 \pi \lvert y \rvert)^j j!} \binom{\kappa}{h} \frac{(\kappa-h)!}{(\kappa-h-2j)!} \left(\frac{y}{\lvert y \rvert}\right)^{\kappa-h} K_{-\frac{1}{2} + \kappa - h - j}\left(2 \pi \lvert y \rvert \right) \\
&= \frac{1}{2} \sum_{h = 0}^\infty \binom{\kappa}{h} \left(\frac{y}{\lvert y \rvert}\right)^{\kappa-h} \lvert y \rvert^{-\frac{1}{2}} e^{-2 \pi \lvert y \rvert}.
\end{align*}
Now the sum over $h$ vanishes for $y < 0$ and is equal to $2^\kappa$ for $y > 0$, which shows the result.
\end{proof}

Now we can calculate the Fourier coefficient for $q(\lambda) \geq 0$ at $s = 0$.

\begin{lem}\label{lem:ThetaLiftFourierExpansions=0}
For $q(\lambda) > 0$ the coefficient $b_{k,\beta}(\lambda, Y)$ is equal to
\begin{align*}
\sum_{b \in \IZ / N_z \IZ} e\left(- \frac{(\lambda, \zeta)}{N_z}\right) \sum_{n \mid \lambda} n^{\kappa - 1} e\left(\frac{nb}{N_z}\right) c_{k,\beta}\left(\frac{\lambda}{n} - \frac{(\lambda, \zeta)}{nN_z} z + \frac{bz}{N_z}, \frac{q(\lambda)}{n^2}\right) e(\lambda, iY),
\end{align*}
if $(\lambda, Y) > 0$ and $b_{k,\beta}(\lambda, Y, 0) = 0$ if $(\lambda, Y) < 0$. For $q(\lambda) = 0$ with $(\lambda, Y) > 0$ the Fourier coefficient $b_{k,\beta}(\lambda, Y)$ is given by
\begin{align*}
& \sum_{b \in \IZ / N_z \IZ} \bigg(e\left(- \frac{(\lambda, \zeta)}{N_z}\right) \sum_{n \mid \lambda} n^{\kappa - 1} e\left(\frac{nb}{N_z}\right) \\
&\times (\delta_{\beta, \frac{\lambda}{n} - \frac{(\lambda, \zeta)}{nN_z} z + \frac{bz}{N_z}} + (-1)^\kappa \delta_{-\beta, \frac{\lambda}{n} - \frac{(\lambda, \zeta)}{nN_z} z + \frac{bz}{N_z}}) e(\lambda, iY) \\
&+ \frac{2 q(Y)^{1 - k}}{\lvert (\lambda, Y) \rvert^{\frac{1}{2} - k}} e\left(- \frac{(\lambda, \zeta)}{N_z}\right) \sum_{n \mid \lambda} n^{1 - 2k + \kappa} c_{k,\beta}\left(\frac{\lambda}{n} - \frac{(\lambda, \zeta)}{nN_z} z + \frac{bz}{N_z}, 0\right) e\left(\frac{nb}{N_z}\right) \\
&\times \sum_{h = 0}^\infty \sum_{j = 0}^\infty \frac{(-1)^{j}}{(4 \pi \lvert(\lambda, Y)\rvert)^j j!} \binom{\kappa}{h} \frac{(\kappa - h)!}{(\kappa-h-2j)!} \\
&\times \left(\frac{(\lambda, Y)}{\lvert (\lambda, Y) \rvert}\right)^{\kappa-h} K_{\frac{1}{2} - k + \kappa - h - j}(2 \pi (\lambda, Y ))\bigg).
\end{align*}
For $q(\lambda) = 0$ with $(\lambda, Y) < 0$, $b_{k,\beta}(\lambda, Y)$ is given by
\begin{align*}
& \frac{2 q(Y)^{1 - k}}{\lvert (\lambda, Y) \rvert^{\frac{1}{2} - k}} \sum_{b \in \IZ / N_z \IZ} e\left(- \frac{(\lambda, \zeta)}{N_z}\right) \sum_{n \mid \lambda} n^{1 - 2k + \kappa} c_{k,\beta}\left(\frac{\lambda}{n} - \frac{(\lambda, \zeta)}{nN_z} z + \frac{bz}{N_z}, 0\right) e\left(\frac{nb}{N_z}\right) \\
&\times \sum_{h = 0}^\infty \sum_{j = 0}^\infty \frac{(-1)^{j}}{(4 \pi \lvert(\lambda, Y)\rvert)^j j!} \binom{\kappa}{h} \frac{(\kappa - h)!}{(\kappa-h-2j)!} \left(\frac{(\lambda, Y)}{\lvert (\lambda, Y) \rvert}\right)^{\kappa-h} K_{\frac{1}{2} - k + \kappa - h - j}(2 \pi \lvert (\lambda, Y ) \rvert).
\end{align*}
\end{lem}

\begin{proof}
Recall that for $q(\lambda) > 0$ the coefficient $b_{k,\beta}(\lambda, Y, s)$ is given by
\begin{align*}
& \sum_{b \in \IZ / N_z \IZ} \bigg(\frac{1}{\sqrt{2}\lvert Y \rvert^{\kappa-1}} \sum_{h = 0}^\infty (2i)^{-h} \sum_{j = 0}^\infty \frac{(-1)^j i^{h}}{(8 \pi)^j j!} \binom{\kappa}{h} \frac{(\kappa-h)!}{(\kappa-h-2j)!} \lvert Y \rvert^h \left(\frac{(\lambda, Y)}{\lvert Y \rvert}\right)^{\kappa-h - 2j} \\
&\times e\left(- \frac{(\lambda, \zeta)}{N_z}\right) \sum_{n \mid \lambda} n^{2j - \kappa + 2h} e\left(\frac{nb}{N_z}\right) c_{k,\beta}\left(\frac{\lambda}{n} - \frac{(\lambda, \zeta)}{nN_z} z + \frac{bz}{N_z}, \frac{q(\lambda)}{n^2}, s\right) \\
&\times \int_0^\infty \exp\left(-\frac{\pi n^2}{2vz_Z^2} - \frac{2 \pi v q_\omega(\lambda)}{n^2}\right) \calW_s\left(4 \pi \frac{q(\lambda)}{n^2}v\right) v^{-\frac{3}{2} + \kappa - h - j} \mathrm{d}v \bigg).
\end{align*}
We now plug in $s = 0$. According to \cite{BruinierKuehn} we have
\begin{align*}
\calW_0(v) = e^{-\frac{v}{2}} \cdot \begin{cases} 1 &\mbox{ if } v > 0, \\ \Gamma(1 - k, -v) &\mbox{ if } v < 0.
\end{cases}
\end{align*}
Now observe that
$$q_\omega(\lambda) + q(\lambda) = 2 q(\lambda_{\omega^+}) = \frac{(\lambda, Y)^2}{2 q(Y)}$$
and use the formula \cite[p. 313, 6.3(17)]{Erdelyi}
\begin{align*}
\int_{v = 0}^\infty \exp\left(- \alpha v - \frac{\beta}{v}\right) v^{\gamma} \mathrm{d}v = 2 \left(\frac{\beta}{\alpha}\right)^{\frac{1}{2} (\gamma + 1)} K_{\gamma + 1}(2\sqrt{\alpha \beta})
\end{align*}
with $\alpha = \frac{\pi (\lambda, Y)^2}{n^2 q(Y)}, \beta = \pi n^2 q(Y), \gamma = -\frac{3}{2} + \kappa - h - j$ to obtain
\begin{align*}
& \int_0^\infty \exp\left(-\frac{\pi n^2 q(Y)}{v} - \frac{2 \pi v q_\omega(\lambda)}{n^2}\right) \calW_0\left(4 \pi \frac{q(\lambda)}{n^2}v\right) v^{-\frac{3}{2} + \kappa - h - j} \mathrm{d}v \\
&= 2 \left(\frac{n^2 q(Y)}{\lvert (\lambda, Y) \rvert}\right)^{-\frac{1}{2} + \kappa - h - j} K_{-\frac{1}{2} + \kappa - h - j}\left(2 \pi \lvert (\lambda, Y) \rvert \right).
\end{align*}
This yields for the Fourier coefficient $b_{k,\beta}(\lambda, Y)$
\begin{align*}
& 2^{1 - \kappa} \lvert (\lambda, Y) \rvert^{\frac{1}{2}} \sum_{b \in \IZ / N_z \IZ} e\left(- \frac{(\lambda, \zeta)}{N_z}\right) \sum_{n \mid \lambda} n^{\kappa - 1} e\left(\frac{nb}{N_z}\right) c_{k,\beta}\left(\frac{\lambda}{n} - \frac{(\lambda, \zeta)}{nN_z} z + \frac{bz}{N_z}, \frac{q(\lambda)}{n^2}\right) \\
&\times \sum_{h = 0}^\infty \sum_{j = 0}^\infty \frac{(-1)^{j}}{(4 \pi \lvert (\lambda, Y) \rvert)^j j!} \binom{\kappa}{h} \frac{(\kappa-h)!}{(\kappa-h-2j)!} \left(\frac{(\lambda, Y)}{\lvert (\lambda, Y) \rvert}\right)^{\kappa-h} K_{-\frac{1}{2} + \kappa - h - j}\left(2 \pi \lvert (\lambda, Y) \rvert \right).
\end{align*}
Now use Lemma \ref{lem:SumOfKBesselFunctions}. For $q(\lambda) = 0$ plug in $s = 0$ and use Lemma \ref{lem:SumOfKBesselFunctions}.
\end{proof}

\begin{defn}
Define the holomorphic part $\Phi_{k,\beta}^+(Z)$ of the theta lift by
\begin{align*}
&\frac{\Gamma(\kappa) N_z^\kappa}{(-2 \pi i)^\kappa}\sum_{b \in \IZ / N_z \IZ} \delta_{\beta, \frac{bz}{N_z}}\zeta^b(\kappa) \\
&+ \sum_{\substack{\lambda \in K' \\ q(\lambda) = 0 \\ (\lambda, Y) > 0}} \sum_{b \in \IZ / N_z \IZ} \sum_{n \mid \lambda} n^{\kappa - 1} e\left(\frac{nb}{N_z}\right) \\
&\times (\delta_{\beta, \frac{\lambda}{n} - \frac{(\lambda, \zeta)}{nN_z} z + \frac{bz}{N_z}} + (-1)^\kappa \delta_{-\beta, \frac{\lambda}{n} - \frac{(\lambda, \zeta)}{nN_z} z + \frac{bz}{N_z}}) e\left(\lambda, Z - \frac{\zeta}{N_z}\right) \\
&+ \sum_{\substack{\lambda \in K' \\ q(\lambda) > 0 \\ (\lambda, Y) > 0}} \sum_{b \in \IZ / N_z \IZ} \sum_{n \mid \lambda} n^{\kappa - 1} e\left(\frac{nb}{N_z}\right) c_{k,\beta}\left(\frac{\lambda}{n} - \frac{(\lambda, \zeta)}{nN_z} z + \frac{bz}{N_z}, \frac{q(\lambda)}{n^2}\right) e\left(\lambda, Z - \frac{\zeta}{N_z}\right)
\end{align*}
and the non-holomorphic part $\Phi_{k,\beta}^-(Z) = \Phi_{k,\beta}(Z) - \Phi_{k,\beta}^+(Z)$.
\end{defn}

\begin{rem}
Assume that $E_{k, \frakv}(\tau) := E_{k, \frakv}(\tau, 0)$ is holomorphic. Then $\Phi_\frakv^-(Z)$ vanishes identically and hence $\Phi_\frakv(Z) = \Phi_\frakv^+(Z)$ is a holomorphic modular form. See also \cite[Theorem 14.3]{Borcherds}. Write $\M_\kappa^\Phi(\Gamma(L))$ for the space of holomorphic modular forms of weight $\kappa$ that are given as a theta lift $\Phi_\frakv(Z)$ for some $\frakv \in \Iso(\IC[L' / L])$.
\end{rem}

\begin{defn}
Define the holomorphic boundary part of the theta lift $\Phi_{k,\beta}^{\partial +}(Z)$ to be
\begin{align*}
&\frac{\Gamma(\kappa) N_z^\kappa}{(-2 \pi i)^\kappa} \sum_{b \in \IZ / N_z \IZ} \delta_{\beta, \frac{bz}{N_z}} \zeta^b(\kappa) \\
&+ \sum_{\substack{\lambda \in K' \\ q(\lambda) = 0 \\ (\lambda, Y) > 0}} \sum_{b \in \IZ / N_z \IZ} \sum_{n \mid \lambda} n^{\kappa - 1} e\left(\frac{nb}{N_z}\right) \\
&\times (\delta_{\beta, \frac{\lambda}{n} - \frac{(\lambda, \zeta)}{nN_z} z + \frac{bz}{N_z}} + (-1)^\kappa \delta_{-\beta, \frac{\lambda}{n} - \frac{(\lambda, \zeta)}{nN_z} z + \frac{bz}{N_z}}) e\left(\lambda, Z - \frac{\zeta}{N_z}\right).
\end{align*}
\end{defn}

\begin{prop}\label{prop:HolomorphicBoundaryPart}
We have
\begin{align*}
\Phi_{k,\beta}^{\partial +}(Z)
&= \frac{\Gamma(\kappa) N_z^\kappa}{(-2 \pi i)^\kappa} \sum_{b \in \IZ / N_z \IZ} \delta_{\beta, \frac{bz}{N_z}} \zeta^b(\kappa) \\
&+ \sum_{\substack{\lambda \in K \\ q(\lambda) = 0 \\ (\lambda, Y) > 0 \\ \lambda \text{ primitive}}} \sum_{\substack{b\in \IZ / N_z \IZ \\ c \in \IZ / N_\lambda \IZ}} \delta_{\beta, \frac{c \lambda}{N_\lambda} - \frac{c (\lambda, \zeta)}{N_\lambda N_z} z + \frac{bz}{N}} \sum_{m = 1}^\infty \tilde{\sigma}_{\kappa - 1}^{c, b}(m) e\left(\frac{m (\lambda, Z - \frac{\zeta_K}{N_z})}{N_\lambda}\right).
\end{align*}
In particular, for an isotropic plane $I \subseteq L \otimes \IQ$ with $I = \langle z, d \rangle$, we have, writing $\beta = \frac{c_\beta d}{N_d} - \frac{c_\beta (d, \zeta)z}{N_d N_z} + \frac{b_\beta z}{N_z}$, (if such a decomposition exists it is unique and if it does not exist then $\Phi_{k,\beta}^{\partial +}\vert_I$ vanishes identically)
\begin{align*}
\Phi_{k,\beta}^{\partial +} \vert_I (\tau)
&= \frac{\Gamma(\kappa) N_z^\kappa}{(-2 \pi i)^\kappa} \delta_{c_\beta} \zeta^{b_\beta}(\kappa) + \sum_{m = 1}^\infty \tilde{\sigma}_{\kappa - 1}^{c_\beta, b_\beta}(m) e\left(\frac{m (\tau - (d, \frac{\zeta_K}{N_z}))}{N_d}\right),
\end{align*}
which is the holomorphic part of an Eisenstein series on the boundary component $I$ (see \cite{DiamondShurman}, \cite{Miyake}). Here for $c \in \IZ / N_\lambda \IZ, b \in \IZ / N_z \IZ$ we have the divisor sum
$$\tilde{\sigma}_{\kappa - 1}^{c, b}(m) = \sum_{\substack{n \mid m \\ \frac{m}{n} \equiv c \bmod{N_\lambda}}} \sgn(n) n^{\kappa - 1} e\left(\frac{nb}{N_z}\right),$$
where the sum is over positive and negative divisors. Observe that the constant term only depends on the image of $\frac{z}{N_z}$ in $L' / L$.
\end{prop}

\begin{proof}
For $\lambda \in K$ primitive let $N_\lambda$ be its level. Then we can rewrite the second summand as
\begin{align*}
&\sum_{\substack{\lambda \in K' \\ q(\lambda) = 0 \\ (\lambda, Y) > 0}} \sum_{b \in \IZ / N_z \IZ} \sum_{n \mid \lambda} n^{\kappa - 1} e\left(\frac{nb}{N_z}\right) \\
&\times (\delta_{\beta, \frac{\lambda}{n} - \frac{(\lambda, \zeta)}{nN_z} z + \frac{bz}{N_z}} + (-1)^\kappa \delta_{-\beta, \frac{\lambda}{n} - \frac{(\lambda, \zeta)}{nN_z} z + \frac{bz}{N_z}}) e\left(\lambda, Z - \frac{\zeta}{N_z}\right) \\
&= \sum_{\substack{\lambda \in K \\ q(\lambda) = 0 \\ (\lambda, Y) > 0 \\ \lambda \text{ primitive}}} \sum_{m = 1}^\infty \sum_{b \in \IZ / N_z \IZ} \sum_{n \mid m} n^{\kappa - 1} e\left(\frac{nb}{N_z}\right)  \\
&\times (\delta_{\beta, \frac{m \lambda}{n N_\lambda} - \frac{(m \lambda, \zeta)}{nN_\lambda N_z} z + \frac{bz}{N_z}} + (-1)^\kappa \delta_{-\beta, \frac{m \lambda}{n N_\lambda} - \frac{(m \lambda, \zeta)}{nN_\lambda N_z} z + \frac{bz}{N_z}}) e\left(\frac{m(\lambda, Z - \frac{\zeta_K}{N_z})}{N_\lambda}\right) \\
&= \sum_{\substack{\lambda \in K \\ q(\lambda) = 0 \\ (\lambda, Y) > 0 \\ \lambda \text{ primitive}}} \sum_{\substack{b\in \IZ / N_z \IZ \\ c \in \IZ / N_\lambda \IZ}} (\delta_{\beta, \frac{c \lambda}{N_\lambda} - \frac{c (\lambda, \zeta)}{N_\lambda N} z + \frac{bz}{N_z}} + (-1)^\kappa \delta_{-\beta, \frac{c \lambda}{N_\lambda} - \frac{c (\lambda, \zeta)}{N_\lambda N_z} z + \frac{bz}{N_z}}) \\
&\times \sum_{m = 1}^\infty \sum_{\substack{n \mid m \\ \frac{m}{n} \equiv c \bmod{N_\lambda}}} n^{\kappa - 1} e\left(\frac{nb}{N_z}\right) e\left(\frac{m(\lambda, Z - \frac{\zeta_K}{N_z})}{N_\lambda}\right) \\
\end{align*}
Summing over positive and negative divisors in the divisor sum we can rewrite the holomorphic boundary part as
\begin{align*}
\Phi_{k,\beta}^{\partial +}(Z)
&= \frac{\Gamma(\kappa) N_z^\kappa}{(-2 \pi i)^\kappa} \sum_{b \in \IZ / N_z \IZ} \delta_{\beta, \frac{bz}{N_z}} \zeta^b(\kappa) \\
&+ \sum_{\substack{\lambda \in K \\ q(\lambda) = 0 \\ (\lambda, Y) > 0 \\ \lambda \text{ primitive}}} \sum_{\substack{b\in \IZ / N_z \IZ \\ c \in \IZ / N_\lambda \IZ}} \delta_{\beta, \frac{c \lambda}{N_\lambda} - \frac{c (\lambda, \zeta)}{N_\lambda N_z} z + \frac{bz}{N_z}} \sum_{m = 1}^\infty \tilde{\sigma}_{\kappa - 1}^{c, b}(m) e\left(\frac{m (\lambda, Z - \frac{\zeta_K}{N_z})}{N_\lambda}\right).
\end{align*}
Let now $I \subseteq L \otimes \IQ$ be an isotropic plane with $I = \langle z, d \rangle, d \in \Iso_0(K)$ and consider the Siegel operator corresponding to this plane. Then
\begin{align*}
\Phi_{k,\beta}^{\partial +} \vert_I (\tau)
&= \frac{\Gamma(\kappa) N_z^\kappa}{(-2 \pi i)^\kappa} \sum_{b \in \IZ / N_z \IZ} \delta_{\beta, \frac{bz}{N_z}} \zeta^b(\kappa) \\
&+ \sum_{\substack{b\in \IZ / N_z \IZ \\ c \in \IZ / N_d \IZ}} \delta_{\beta, \frac{c d}{N_d} - \frac{c (d, \zeta)}{N_d N_z} z + \frac{bz}{N_z}} \sum_{m = 1}^\infty \tilde{\sigma}_{\kappa - 1}^{c, b}(m) e\left(\frac{m (\tau - (d, \frac{\zeta_K}{N_z}))}{N_d}\right).
\end{align*}
Writing $\beta = \frac{c_\beta d}{N_d} - \frac{c_\beta (d, \zeta)z}{N_d N_z} + \frac{b_\beta z}{N_z}$ (if such a decomposition exists it is unique and if it does not exist then $\Phi_{k,\beta}^{\partial +}\vert_I$ vanishes identically), we obtain
\begin{align*}
\Phi_{k,\beta}^{\partial +} \vert_I (\tau)
&= \frac{\Gamma(\kappa) N_z^\kappa}{(-2 \pi i)^\kappa} \delta_{c_\beta} \zeta^{b_\beta}(\kappa) + \sum_{m = 1}^\infty \tilde{\sigma}_{\kappa - 1}^{c_\beta, b_\beta}(m) e\left(\frac{m (\tau - (d, \frac{\zeta_K}{N_z}))}{N_d}\right),
\end{align*}
which shows the result.
\end{proof}

\begin{prop}
Assume that the map $\pi$ is surjective. Then the theta lift is injective.
\end{prop}

\begin{proof}
Assume that there is some $\frakv \in \Iso(\IC[L' / L])$ with $\Phi_\frakv(Z) = 0$. Then in particular, $\Phi_\frakv^{\partial +}(Z) = 0$ for every $0$-dimensional cusp $z$, i.e.
$$\sum_{b \in \IZ / N_z \IZ} \frakv_{\frac{bz}{N_z}} \zeta^b(\kappa) = 0$$
for all $0$-dimensional cusps $z$. We want to show that $E_{k, \frakv}(\tau, 0) = 0$, which is equivalent to $\frakv + (-1)^\kappa \frakv^* = 0$. Therefore, assume there is some $\delta \in \Iso(L' / L)$ with $\frakv_\delta \neq -(-1)^\kappa \frakv_{-\delta}$. By surjectivity of $\pi$ there is a $0$-dimensional cusp $z$ corresponding to $\delta$. Choose such $\delta$ with minimal order. Then by assumption the value in the $0$-dimensional cusp $z$ is
$$\sum_{b \in (\IZ / N_z \IZ)^\times} \frakv_{b \delta} \zeta^b(\kappa) = 0.$$
But of course this is also true for the $0$-dimensional cusps corresponding to $c \delta$ for $c \in (\IZ / N_z \IZ)^\times$, i.e. we have
$$\sum_{b \in (\IZ / N_z \IZ)^\times} \frakv_{b c \delta} \zeta^b(\kappa) = \sum_{b \in (\IZ / N_z \IZ)^\times} \frakv_{b \delta} \zeta^{bc^*}(\kappa) = 0.$$
Rewrite this using $\zeta^{bc^*}(\kappa) = \zeta_+^{bc^*}(\kappa) + (-1)^\kappa \zeta_+^{-bc^*}(\kappa)$ to obtain
$$\sum_{b \in (\IZ / N_z \IZ)^\times} (\frakv_{b \delta} + (-1)^\kappa \frakv_{-b \delta}) \zeta_+^{bc^*}(\kappa) = 0.$$
For a character $\chi : (\IZ / N_z \IZ)^\times \to \IC^\times$ consider now
\begin{align*}
0 &= \sum_{c \in (\IZ / N_z \IZ)^\times} \chi(c) \sum_{b \in (\IZ / N_z \IZ)^\times} (\frakv_{b \delta} + (-1)^\kappa \frakv_{-b \delta}) \zeta_+^{bc^*}(\kappa) \\
&= \sum_{b \in (\IZ / N_z \IZ)^\times} \chi(b) (\frakv_{b \delta} + (-1)^\kappa \frakv_{-b \delta}) \sum_{c \in (\IZ / N_z \IZ)^\times} \chi(c) \zeta_+^{c^*}(\kappa) \\
&= L(\overline{\chi}, \kappa) \sum_{b \in (\IZ / N_z \IZ)^\times} \chi(b) (\frakv_{b \delta} + (-1)^\kappa \frakv_{-b \delta}).
\end{align*}
Since $L(\overline{\chi}, \kappa) \neq 0$ we have
$$\sum_{b \in (\IZ / N_z \IZ)^\times} \chi(b) (\frakv_{b \delta} + (-1)^\kappa \frakv_{-b \delta}) = 0$$
for all Dirichlet characters $\chi$. But that means $\frakv_{b \delta} + (-1)^\kappa \frakv_{-b\delta} = 0$ for all $b \in (\IZ / N_z \IZ)^\times$ contradicting the assumption $\frakv_{\delta} \neq -(-1)^\kappa \frakv_{-\delta}$.
\end{proof}

Of course, if $\pi$ is not surjective, then the theta lift is not injective, i.e. the converse of the previous theorem is also true.

\begin{prop}
For $\kappa$ even and every $\delta \in L' / L$ isotropic there is a theta lift $\Phi_{k, \frakv}$ for some $\frakv \in \Iso(\IC[L' / L])$ such that the holomorphic part vanishes in all $0$-dimensional cusps except for the $0$-dimensional cusps corresponding to $\pm \delta$. For $\kappa$ odd this is true if $\delta \neq - \delta$. Moreover, in this case we have $\Phi_{k, \frakv} = \calG_{\kappa, \delta}$.
\end{prop}

\begin{proof}
If $\delta$ generates a maximal cyclic isotropic subgroup, then the holomorphic part of $\Phi_{c\delta}, c \in (\IZ / N_\delta \IZ)^\times$ vanishes in every $0$-dimensional cusp except for the $0$-dimensional cusps corresponding to the generators of $\langle \delta \rangle$. In the $0$-dimensional cusps corresponding to $b \delta$ for $b \in (\IZ / N_\delta \IZ)^\times$ the value is given by
$$\frac{\Gamma(\kappa) N_\delta^\kappa}{(-2 \pi i)^\kappa} \zeta^{c b^*}(\kappa).$$
Now consider the linear combination
$$\frac{(-2 \pi i)^\kappa}{2 \varphi(N_\beta) \Gamma(\kappa) N_\delta^\kappa}\sum_{\chi} \frac{1}{L(\chi, \kappa)} \sum_{c \in (\IZ / N_\delta \IZ)^\times} \chi(c) \Phi_{c \delta},$$
whose value in the $0$-dimensional cusp $b \delta$ for $b \in (\IZ / N_\delta \IZ)^\times$ is given by
\begin{align*}
&\frac{1}{2 \varphi(N_\beta)} \sum_{\chi} \frac{1}{L(\chi, \kappa)} \sum_{c \in (\IZ / N_\delta \IZ)^\times} \chi(c) \zeta^{cb^*}(\kappa) \\
&= \frac{1}{2 \varphi(N_\beta)} \sum_{\chi} \frac{\chi(b)}{L(\chi, \kappa)} \sum_{c \in (\IZ / N_\delta \IZ)^\times} \chi(c) \zeta^{c}(\kappa)
= \frac{1}{2 \varphi(N_\beta)} \sum_{\chi} \left(\chi(b) + (-1)^\kappa \chi(-b) \right),
\end{align*}
i.e. the value in the $0$-dimensional cusps vanishes except for the $0$-dimensional cusps corresponding to $\delta$, where the value is $1$. Moreover, using Theorem \ref{thm:ThetaLiftIsEisensteinSeries} shows that this linear combination is in fact given by $\calG_{\kappa, \delta}$. Now do induction over the maximal length of chains of cyclic isotropic subgroups containing $\delta$.
\end{proof}

This means in particular that for $F \in \M_\kappa^\pi(\Gamma(L))$ there is a theta lift $\Phi_\frakv$ for some $\frakv \in \Iso(\IC[L' / L])$ whose holomorphic boundary part is given by the boundary part of $F$ (observe that for $\kappa$ odd the values in $0$-dimensional cusps corresponding to $\delta \in \Iso(L' / L)$ with $\delta = - \delta$ must be zero).

\begin{thm}\label{thm:HolomorphicEisensteinSeriesk>2}
Let $k > 2$ and thus $\kappa > \frac{l}{2} + 1$. Then $\Phi_{k,\beta}(Z) = \Phi_{k,\beta}^+(Z)$ is a holomorphic modular form of weight $\kappa$ which is an Eisenstein series on the boundary. In particular we have $\M_\kappa^\pi(\Gamma(L)) = \S_\kappa(\Gamma(L)) + \M_\kappa^\Phi(\Gamma(L))$ in this case.
\end{thm}

\begin{proof}
Using that the coefficients $c_{k,\beta}(\gamma, n, 0)$ vanish for $n \leq 0$ one obtains the result using Theorem \ref{thm:ThetaLiftFourierExpansion} and Lemma \ref{lem:ThetaLiftFourierExpansions=0} together with \cite[Theorem 10.3]{Borcherds} for the term $\Phi_{k, \beta}^K(Y / \lvert Y \rvert)$. Of course, this reproduces the result of \cite[Theorem 14.3]{Borcherds}.
\end{proof}

We obtain the Fourier expansion of the Eisenstein series $\calE_{\kappa, \frakv}(Z)$ and that the Eisenstein series $\calE_{\kappa, \frakv}(Z) = \calE_{\kappa, \frakv}(Z, 0)$ for $\frakv \in \Iso(\IC[L' / L])$ are holomorphic if $\kappa > \frac{l}{2} + 1$. Moreover, if
$$\pi_L : \Gamma(L) \bs \Iso_0(L') \to \Iso(L' / L)$$
is injective, then we obtain all holomorphic orthogonal Eisenstein series as a lift of vector-valued Eisenstein series $E_{k, \frakv}(\tau)$.

If $k = 0$, the Eisenstein series $E_{k, \frakv}(\tau) = E_{k, \frakv}(\tau, 0)$ are usually not holomorphic in $\tau$. Hence we can not expect that $\calE_{\kappa, \frakv}(Z) = \calE_{\kappa,\frakv}(Z, 0)$ is holomorphic. Since $\res_{s = 1} E_{k, \frakv}(\tau, s)$ is always an invariant vector we have

\begin{thm}\label{thm:ResidueAts=1}
Let $k = 0, \kappa = \frac{l}{2} - 1 > 1$, i.e. $l > 4$. In the $0$-dimensional cusp $z$ we have the expansion
\begin{align*}
&\res_{s = 1} \Phi_{0,\beta}(Z, s) = \Phi(Z, \res_{s = 1} E_{0, \beta}(\cdot, s)) \\
&= \frac{\Gamma(\kappa) N_z^\kappa}{(-2 \pi i)^\kappa}\sum_{b \in \IZ / N_z \IZ} \res_{s = 1} c_{0,\beta}\left(\frac{bz}{N_z}, 0, s\right) \zeta^b(\kappa) \\
&+ \sum_{\substack{\lambda \in K' \\ q(\lambda) = 0 \\ (\lambda, Y) > 0}} e\left(- \frac{(\lambda, \zeta)}{N_z}\right) \sum_{b \in \IZ / N_z \IZ} \sum_{n \mid \lambda} n^{\kappa - 1} e\left(\frac{nb}{N_z}\right) \\
&\times \res_{s = 1} c_{0,\beta}\left(\frac{\lambda}{n} - \frac{(\lambda, \zeta)}{nN_z} z + \frac{bz}{N_z}, 0, s\right) e(\lambda, Z).
\end{align*}
In particular, every invariant vector yields a holomorphic modular form of singular weight which is an Eisenstein series on the boundary. If $\kappa = 1, l = 4$ we obtain the additional summand $\res_{s = 1} \Phi_{k,\beta}^K(Y / \lvert Y \rvert, s)$, which can be shown to be a constant.
\end{thm}

\begin{proof}
One observes that the terms with $q(\lambda) \neq 0$ are holomorphic in $s = 1$ and hence their residue vanishes (this follows from the corresponding result for vector-valued Eisenstein series). The calculation for the other Fourier coefficients is analogous to the case for $s = 0$. For the term $\res_{s = 1} \Phi_{k, \beta}^K(Y / \lvert Y \rvert, s)$ see \cite[Theorem 10.3]{Borcherds}. See also \cite[Theorem 14.3]{Borcherds}.
\end{proof}

The question if this yields all holomorphic modular forms of singular weight which are Eisenstein series on the boundary will be answered in the next section. We want to mention that we can also construct these in a different way. For an invariant vector $\frakv \in \IC[L' / L]$ we have
$$E_k(\tau, s) \frakv = \sum_{\substack{\beta \in L' / L \\ q(\beta) = 0}} \frakv_\beta E_{k, \beta}(\tau, s),$$
where $E_k(\tau, s)$ is the usual suitably normalized scalar-valued Eisenstein series for $\SL_2(\IZ)$. Now the left-hand-side is holomorphic in $s = 0$ and equal to a multiple of $\frakv$ and the lift of $E_k(\tau, s) \frakv$ is holomorphic in $s = 0$ and yields a holomorphic modular form for $s = 0$ as in the case $k > 2$.

\section{Lifting Holomorphic Orthogonal Modular Forms}\label{sec:LiftingOrthogonal}

Let $L$ be an even lattice of signature $(2, l), l \geq 3$. For $z \in \Iso_0(L)$ and $z' \in L'$ with $(z, z') = 1$ write $K = K_z = L \cap z^\perp \cap z'^\perp$. Let $b_1, \ldots, b_l$ be a basis of $K \otimes \IR$ with $b_1 \perp \langle b_2, \ldots, b_l \rangle$ and $q(b_1) > 0$. If $Z = z_1 b_1 + z_2 b_2 + z_3 b_3 + \ldots + z_l b_l \in K \otimes \IC$, we write $Z = (z_1, \ldots, z_l)$ and similarly $X = (x_1, \ldots, x_l), Y = (y_1, \ldots, y_l)$ if $Z = X + i Y$ with $X, Y \in K \otimes \IR$. Denote by $\IH_l = K \otimes \IR + i C$ the corresponding tube domain model, where
$$C = \{Y = (y_1, \ldots, y_l) \in K \otimes \IR \mid y_1 > 0, q(Y) > 0\}.$$
For $\lambda \in \Iso_0(L')$ let $\sigma_\lambda \in O^+(V)$ with $\sigma_\lambda \lambda = z$ and set $\lambda' = \sigma_\lambda^{-1} z'$. Define
$$K_\lambda = \lambda^\perp \cap \lambda'^\perp \cap L = L \cap \sigma_\lambda^{-1}(K \otimes \IR).$$
Then
$$\sigma_\lambda \Gamma(L)_\lambda \sigma_\lambda^{-1} = \Gamma(\sigma_\lambda L)_z \supseteq (\sigma_\lambda K_\lambda) \rtimes \Gamma(\sigma_\lambda K_\lambda),$$
where $\sigma_\lambda K_\lambda$ acts via translation and $\Gamma(\sigma_\lambda K_\lambda)$ via multiplication on $\IH_l = K_z \otimes \IR + iC$. A fundamental domain is given by $\calF = \calF_1 + i \calF_2$, where $\calF_1$ is a fundamental domain of the action $\sigma_\lambda K_\lambda$ on $K_z \otimes \IR$ and $\calF_2$ is a fundamental domain of the action $\Gamma(\sigma_\lambda K_\lambda)$ on $C$. By abuse of notation we will write $\sigma_\lambda K_\lambda \bs K_z \otimes \IR + i \Gamma(\sigma_\lambda K_\lambda) \bs C$. Moreover, recall the map
$$\pi_L : \Gamma(L) \bs \Iso_0(L') \to \Iso(L' / L).$$

Let $F : \IH_l \to \IC$ be a modular form of weight $\kappa$. We define its theta lift to be (if it exists)
$$\Phi^*(\tau, F) := \int_{\Gamma(L) \bs \IH_l} F(Z) \Theta_L(\tau, Z) q(Y)^\kappa \frac{\mathrm{d}X \mathrm{d}Y}{q(Y)^l}.$$
Proposition \ref{prop:ThetaSiegelDomainGrowth} and Lemma \ref{lem:SiegelDomainIntegrable} show that the theta lift exists for holomorphic modular forms of singular weight (in fact the lift exists for weight $\kappa < l - 1$ and for arbitrary weight if $F$ is a cusp form). A straight forward calculation yields

\begin{lem}\label{lem:ThetaLiftsAdjoint}
Let $F : \IH_l \to \IC$ be a holomorphic modular form of singular weight $\kappa = \frac{l}{2} - 1$ and $\frakv \in \Inv(\IC[L' / L])$ an invariant vector. Then
\begin{align*}
\langle \Phi^*(\tau, F), \frakv \rangle = \langle F, \Phi(Z, \frakv) \rangle,
\end{align*}
where the left hand side denotes the Petersson inner product on invariant vectors and the right hand side denotes the Petersson inner product on holomorphic modular forms of singular weight. In particular, the theta lifts are adjoint to eachother.
\end{lem}

\begin{lem}\label{lem:AdjointThetaLiftHarmonic}
Let $l \geq 3, \kappa = \frac{l}{2} - 1$. If $F$ and $\Omega_\kappa F$ are square-integrable, then the theta lift exists and we have $\Phi^*(\tau, \Omega_\kappa F) = \Delta_0 \Phi^*(\tau, F)$. In particular, if $F$ is holomorphic, then $\Phi^*(\tau, F)$ exists and is harmonic.
\end{lem}

\begin{proof}
Let $\Gamma \subseteq \Gamma(L)$ be a finite index subgroup which acts freely. Then $\Gamma \bs \IH_l$ is a complete connected hermitian manifold. By Corollary \ref{cor:ThetaSquareIntegrable}, $\Theta_L$ and $\Omega_\kappa \Theta_L$ are square-integrable. Using the assumption on $F$ and $\Omega_\kappa F$, we can apply Theorem \ref{thm:LaplaceSelfAdjoint} to square-integrable sections of the hermitian line bundle of modular forms of weight $\kappa$. This yields, using $\overline{\Omega_\kappa \overline{\Theta_L}} = \Delta_0 \Theta_L$ and Theorem \ref{thm:LaplaceSelfAdjoint},
\begin{align*}
\Delta_0 \Phi^*(\tau, F) &= \Delta_0 \int_{\Gamma(L) \bs \IH_l} F(Z) \Theta_L(\tau, Z) q(Y)^\kappa \frac{\mathrm{d}X \mathrm{d}Y}{q(Y)^l} \\
&= [\Gamma(L) : \Gamma] \Delta_0 \int_{\Gamma \bs \IH_l} F(Z) \Theta_L(\tau, Z) q(Y)^\kappa \frac{\mathrm{d}X \mathrm{d}Y}{q(Y)^l} \\
&= [\Gamma(L) : \Gamma] \int_{\Gamma \bs \IH_l} F(Z) \Delta_0 \Theta_L(\tau, Z) q(Y)^\kappa \frac{\mathrm{d}X \mathrm{d}Y}{q(Y)^l} \\
&= [\Gamma(L) : \Gamma] \int_{\Gamma \bs \IH_l} F(Z) \overline{\Omega_\kappa \overline{\Theta_L(\tau, Z)}} q(Y)^\kappa \frac{\mathrm{d}X \mathrm{d}Y}{q(Y)^l} \\
&= [\Gamma(L) : \Gamma] \int_{\Gamma \bs \IH_l} \Omega_\kappa F(Z) \Theta_L(\tau, Z) q(Y)^\kappa \frac{\mathrm{d}X \mathrm{d}Y}{q(Y)^l} \\
&= \int_{\Gamma(L) \bs \IH_l} \Omega_\kappa F(Z) \Theta_L(\tau, Z) q(Y)^\kappa \frac{\mathrm{d}X \mathrm{d}Y}{q(Y)^l} = \Phi^*(\tau, \Omega_\kappa F).
\end{align*}
The other assertions follow immediately.
\end{proof}

\begin{thm}\label{thm:AdjointThetaLiftInvariantVector}
Let $F$ be a holomorphic modular form of singular weight $\kappa = \frac{l}{2} - 1 > 0$. Then its theta lift
$$\Phi^*(\tau, F) = \int_{\Gamma(L) \bs \IH_l} F(Z) \Theta_L(\tau, Z) q(Y)^\kappa \frac{\mathrm{d}X \mathrm{d}Y}{q(Y)^l}$$
is an invariant vector given by
$$\frac{\Gamma(l/2)}{2(2 \pi)^{l / 2}} \sum_{\gamma \in \Iso(L' / L)} \sum_{\substack{\delta \in \Iso(L' / L) \\ \gamma = k_\delta \delta}} \zeta_+^{k_\delta}(l - \kappa) \sum_{\lambda \in \pi_L^{-1}(\delta)} a_{F, \lambda}(0) C(\lambda) \frake_\gamma,$$
where
\begin{align*}
C(\lambda)
&= \vol(\sigma_\lambda K_\lambda) [\Gamma(\sigma_\lambda L)_z : \sigma_\lambda K_\lambda \rtimes \Gamma(\sigma_\lambda K_\lambda)] C(\Gamma(\sigma_\lambda K_\lambda)) \\
&= \vol(K_\lambda) [\Gamma(L)_\lambda : K_\lambda \rtimes \Gamma(K_\lambda)] C(\Gamma(K_\lambda))
\end{align*}
for $\sigma_\lambda \in O^+(V)$ with $\sigma_\lambda \lambda = z$ and $C(\Gamma(K_\lambda))$ is a positive constant.
\end{thm}

\begin{proof}
First observe that the integral converges since holomorphic modular forms of singular weight are square integrable. We have
\begin{align*}
&\Phi^*(\tau, F) = \int_{\Gamma(L) \bs \IH_l} F(Z) \Theta_L(\tau, Z) q(Y)^\kappa \frac{\mathrm{d}X \mathrm{d}Y}{q(Y)^l} \\
&= \frac{v^{\frac{l}{2}}}{2} \int_{\Gamma(L) \bs \IH_l} F(Z) \sum_{\lambda \in L'} \frac{(\lambda, Z_L)^{\kappa}}{q(Y)^{\kappa}} \exp(-2 \pi v q_{Z_L}(\lambda)) \frake_\lambda(u q(\lambda)) q(Y)^\kappa \frac{\mathrm{d}X \mathrm{d}Y}{q(Y)^l} \\
&= \frac{v^{\frac{l}{2}}}{2} \sum_{\lambda \in L'} \int_{\Gamma(L) \bs \IH_l} F(Z) (\lambda, Z_L)^{\kappa} \exp(-2 \pi v q_{Z_L}(\lambda)) \frac{\mathrm{d}X \mathrm{d}Y}{q(Y)^l} \frake_\lambda(u q(\lambda))
\end{align*}
and remark that the integral is bounded for $v \to \infty$ and hence the lift grows polynomially. Moreover, by Lemma \ref{lem:AdjointThetaLiftHarmonic} it is harmonic of weight $0$ and thus its growth comes from the constant Fourier coefficient. The constant Fourier coefficients are then given by (the $\lambda = 0$ term vanishes since $\kappa > 0$)
\begin{align*}
&\frac{v^{\frac{l}{2}}}{2} \sum_{\substack{\lambda \in \Gamma(L) \bs L' \\ q(\lambda) = 0}} \int_{\Gamma(L)_\lambda \bs \IH_l} F(Z) (\lambda, Z_L)^{\kappa} \exp(-2 \pi v q_{Z_L}(\lambda))\frac{\mathrm{d}X \mathrm{d}Y}{q(Y)^l} \frake_\lambda \\
&= \frac{v^{\frac{l}{2}}}{2} \sum_{\lambda \in \Gamma(L) \bs \Iso_0(L')} \sum_{m = 1}^\infty m^\kappa \int_{\Gamma(L)_\lambda \bs \IH_l} F(Z) (\lambda, Z_L)^{\kappa} \exp(-2 \pi v m^2 q_{Z_L}(\lambda)) \frac{\mathrm{d}X \mathrm{d}Y}{q(Y)^l} \frake_{m\lambda}.
\end{align*}
As above let $\sigma_{\lambda} \in O^+(V)$ such that $\sigma_{\lambda} \lambda = z$ and write $\lambda' = \sigma_{\lambda}^{-1}(z')$. Then we can rewrite the integral to (observe that $(z, Z_L) = 1$ and $q_{Z_L}(z) = 1 / q(Y)$)
\begin{align*}
\int_{\Gamma(\sigma_\lambda L)_z \bs \IH_l} (F \mid_\kappa \sigma_\lambda)(Z) \exp(-2 \pi v m^2 / q(Y)) \frac{\mathrm{d}X \mathrm{d}Y}{q(Y)^l}.
\end{align*}
and hence, using the Fourier expansion of $F \mid_k \sigma_\lambda$
\begin{align*}
&\sum_{\delta \in \sigma_\lambda K_\lambda'} a_{F, \lambda}(\delta) \int_{\Gamma(\sigma_\lambda L)_z \bs \IH_l} e(\delta, Z) \exp(-2 \pi v m^2 / q(Y)) \frac{\mathrm{d}X \mathrm{d}Y}{q(Y)^l} \\
&= [\Gamma(L)_\lambda : K_\lambda \rtimes \Gamma(K_\lambda)] \sum_{\delta \in \sigma_\lambda K_\lambda'} a_{F, \lambda}(\delta) \\
&\times \int_{\sigma_\lambda K_\lambda \rtimes \Gamma(\sigma_\lambda K_\lambda) \bs \IH_l} e(\delta, Z) \exp(-2 \pi v m^2 / q(Y)) \frac{\mathrm{d}X \mathrm{d}Y}{q(Y)^l} \\
&= [\Gamma(L)_\lambda : K_\lambda \rtimes \Gamma(K_\lambda)] \sum_{\delta \in \sigma_\lambda K_\lambda'} a_{F, \lambda}(\delta) \\
&\times \int_{\sigma_\lambda K_\lambda \bs K_z \otimes \IR} e(\delta, X) \mathrm{d}X \int_{\Gamma(\sigma_\lambda K_\lambda) \bs C} e(\delta, iY)\exp(-2 \pi v m^2 / q(Y)) \frac{ \mathrm{d}Y}{q(Y)^l} \\
&= a_{F, \lambda}(0) \vol(K_\lambda) [\Gamma(L)_\lambda : K_\lambda \rtimes \Gamma(K_\lambda)] \int_{\Gamma(\sigma_\lambda K_\lambda) \bs C} \exp(-2 \pi v m^2 / q(Y)) \frac{\mathrm{d}X \mathrm{d}Y}{q(Y)^l}.
\end{align*}
Consider the diffeomorphism
\begin{align*}
\varphi : [0, \infty) \times \IR^{l - 1} \to C, \quad (r, y_2, \ldots, y_l) \mapsto \sqrt{r} (\sqrt{1 - q(0, y_2, \ldots, y_l)}, y_2, \ldots, y_l)
\end{align*}
with inverse
\begin{align*}
\varphi^{-1} : C \to [0, \infty) \times \IR^{l - 1}, \quad Y \mapsto (q(Y), y_2 / \sqrt{q(Y)}, \ldots, y_l / \sqrt{q(Y)}).
\end{align*}
Then we have for integrable $f : C \to \IC$
\begin{align*}
\int_{C} f(Y) \mathrm{d}Y &= \int_0^\infty \int_{\IR^{l-1}} f(\varphi(r, y_2, \ldots, y_l)) \lvert \det(\varphi'(r, y_1, \ldots, y_l)) \rvert \mathrm{d}y_2 \ldots \mathrm{d}y_l \mathrm{d}r \\
&= \int_{\IR^{l-1}} \int_0^\infty r^{\frac{l}{2}} f(\varphi(r, y_2, \ldots, y_l)) \frac{\mathrm{d}r}{r} \lvert \det(\varphi'(1, y_1, \ldots, y_l)) \rvert \mathrm{d}y_2 \ldots \mathrm{d}y_l.
\end{align*}
Here $f(Y) = g(q(Y))$ for $g(r) = \exp(2 \pi v n^2 / r) r^{-l}$ and thus
$$f(\varphi(r, y_2, \ldots, y_l)) = g(r) = \exp(2 \pi v n^2 / r) r^{-l}.$$
Hence we obtain
\begin{align*}
& a_{F, \lambda}(0) \vol(K_\lambda) [\Gamma(L)_\lambda : K_\lambda \rtimes \Gamma(K_\lambda)]\\
&\times \int_0^\infty \exp(-2 \pi v m^2 / r) r^{-\frac{l}{2}} \frac{\mathrm{d}r}{r} \int_{\Gamma(\sigma_\lambda K_\lambda) \bs \IR^{l-1}} \lvert \det(\varphi'(1, y_2, \ldots, y_l)) \rvert \mathrm{d}y_2, \ldots, y_l \\
&= a_{F, \lambda}(0) C(\lambda) (2 \pi v)^{- \frac{l}{2}} m^{-l} \Gamma(l/2).
\end{align*}
Hence the constant Fourier coefficient is given by
\begin{align*}
&\frac{\Gamma(l/2)}{2(2 \pi)^{l / 2}} \sum_{\lambda \in \Gamma(L) \bs \Iso_0(L')} a_{F, \lambda}(0) C(\lambda) \sum_{m = 1}^\infty m^{\kappa - l} \frake_{m \lambda} \\
&= \frac{\Gamma(l/2)}{2(2 \pi)^{l / 2}} \sum_{\gamma \in \Iso(L' / L)} \sum_{\substack{\delta \in \Iso(L' / L) \\ \gamma = k_\delta \delta}} \zeta_+^{k_\delta}(l - \kappa) \sum_{\lambda \in \pi_L^{-1}(\delta)} a_{F, \lambda}(0) C(\lambda) \frake_\gamma.
\end{align*}
In particular, this is independent of $v$. Hence $\Phi^*(\tau, F)$ is bounded and since it is harmonic of weight $0$ it is an invariant vector.
\end{proof}

\begin{rem}
Observe that $\Phi^*(\tau, F)$ only depends on the values in the $0$-dimensional cusps. In particular it vanishes on functions that are zero in every $0$-dimensional cusp.
\end{rem}

\begin{rem}
Let $F : \IH_l \to \IC$ be a cusp form of arbitrary weight. Then $\Phi^*(\tau, F)$ exists and one can show as above that it is harmonic. Moreover, the same calculation as in the previous proof shows that the constant Fourier coefficient vanishes. Hence, $\Phi^*(\tau, F)$ decays exponentially and hence is square-integrable and harmonic of weight $k$. Thus, $\Phi^*(\tau, F)$ is a holomorphic cusp form. This reproduces the result of \cite{Oda}.
\end{rem}

\begin{cor}\label{cor:AdjointThetaLiftKernel}
For $\delta \in L' / L$ isotropic let $a_{F, \delta}(0) := \sum_{\lambda \in \pi_L^{-1}(\delta)} a_{F, \lambda}(0) C(\lambda)$. Then $\Phi^*(\tau, F)$ vanishes if and only if $a_{F, \delta}(0)$ vanishes for all isotropic $\delta \in L' / L$. In particular, if $\pi_L$ is injective, then the theta lift $\Phi^*(\tau, F)$ vanishes if and only if $F$ vanishes in every $0$-dimensional cusp.
\end{cor}

\begin{proof}
If $\Phi^*(\tau, F)$ vanishes, then the coefficient of $\frake_\gamma$, given by
\begin{align*}
\sum_{\substack{\delta \in L' / L \\ q(\delta) = 0 \\ \gamma = k_\delta \delta}} \zeta_+^{k_\delta}(l - \kappa) a_{F, \delta}(0),
\end{align*}
vanishes for all isotropic $\gamma \in L' / L$. In particular, if $\gamma \in L' / L$ is isotropic with maximal order such that $a_{F, \gamma}(0) \neq 0$, then for $k' \in (\IZ / N_\gamma \IZ)^\times$ the coefficient of $\frake_{k' \gamma}$
\begin{align*}
\sum_{k \in (\IZ / N_\gamma \IZ)^\times} \zeta_+^{k^*}(l - \kappa) a_{F, k k'\gamma}(0) = \sum_{k \in (\IZ / N_\gamma \IZ)^\times} \zeta_+^{k'k^*}(l - \kappa) a_{F, k \gamma}(0)
\end{align*}
vanishes. Thus, for all Dirichlet characters $\chi : (\IZ / N_\gamma \IZ)^\times \to \IC$, the sum
\begin{align*}
&\sum_{k' \in (\IZ / N_\gamma \IZ)^\times} \chi(k') \sum_{k \in (\IZ / N_\gamma \IZ)^\times} \zeta_+^{k'k^*}(l - \kappa) a_{F, k \gamma}(0) \\
&= \sum_{k' \in (\IZ / N_\gamma \IZ)^\times} \chi(k')\zeta_+^{k'}(l - \kappa) \sum_{k \in (\IZ / N_\gamma \IZ)^\times} \chi(k) a_{F, k \gamma}(0) \\
&= L(l - \kappa, \chi) \sum_{k \in (\IZ / N_\gamma \IZ)^\times} \chi(k) a_{F, k \gamma}(0)
\end{align*}
vanishes. Now $L(l - \kappa, \chi) \neq 0$ and hence
\begin{align*}
\sum_{k \in (\IZ / N_\gamma \IZ)^\times} \chi(k) a_{F, k \gamma}(0) = 0
\end{align*}
for all Dirichlet characters $\chi$. But then we must have $a_{F, k\gamma}(0) = 0$ for all $k$, since Dirichlet characters form an orthogonal basis of $(\IZ / N_\gamma \IZ)^\times$.
\end{proof}

\begin{cor}\label{cor:ThetaLiftImage}
The theta lift $\Phi$ surjects onto the space of holomorphic modular forms of singular weight which are Eisenstein series on the boundary and whose value in a $0$-dimensional cusp only depends on its image in $L' / L$, i.e. we have $\M_\kappa^\Phi(\Gamma(L)) = \M_\kappa^\pi(\Gamma(L))$.
\end{cor}

\begin{proof}
Let $F : \IH_l \to \IC$ be a holomorphic modular form which is an Eisenstein series on the boundary whose value in a $0$-dimensional cusp only depends on its image in $L' / L$. Since $\Phi$ and $\Phi^*$ are adjoint to each other by Lemma \ref{lem:ThetaLiftsAdjoint}, we can write $F = \Phi_{0, \frakv}(Z) + G$ for an invariant vector $\frakv$ and a holomorphic modular form $G : \IH_l \to \IC$ of singular weight with $\Phi^*(\tau, G) = 0$. By Corollary \ref{cor:AdjointThetaLiftKernel} we have $a_{G, \delta}(0) = 0$ for all isotropic $\delta \in L' / L$. Moreover, the value of $G = F - \Phi_{0, \frakv}(Z)$ in a $0$-dimensional cusp only depends on its image in $L' / L$. Hence we have
$$0 = a_{G, \delta}(0) = \sum_{\lambda \in \pi_L^{-1}(\delta)} a_{G, \lambda}(0) C(\lambda) = a_{G, \tilde{\lambda}}(0) \sum_{\lambda \in \pi_L^{-1}(\delta)} C(\lambda)$$
for some $\tilde{\lambda} \in \pi_L^{-1}(\delta)$. But then $a_{G, \tilde{\lambda}}(0) = 0$ and hence $G$ vanishes in every $0$-dimensional cusp and thus is a cusp form on the boundary. But since $F$ and $\Phi_{0, \frakv}(Z)$ are Eisenstein series on the boundary, $G = F - \Phi_{0, \frakv}(Z)$ must vanish on the boundary and hence $G$ is a cusp form. Since we are in singular weight, $G$ must vanish and thus $F = \Phi_{0, \frakv}(Z)$.
\end{proof}

\begin{cor}\label{cor:ThetaLiftSurjective}
If $\pi_L$ is injective, then the theta lift $\Phi$ surjects onto the space of holomorphic modular forms of singular weight which are Eisenstein series on the boundary, i.e. we have $\M_\kappa^\Phi(\Gamma(L)) = M_\kappa^{\partial \Eis}(\Gamma(L))$.
\end{cor}

\begin{proof}
Since $\pi_L$ is injective we have $M_\kappa^\pi(\Gamma(L)) = M_\kappa^{\partial \Eis}(\Gamma(L))$.
\end{proof}

\begin{cor}\label{cor:ThetaLiftMaximalLattice}
If $L$ is a maximal lattice of Witt rank $2$, then the space $\M_\kappa^{\partial \Eis}(\Gamma(L))$ is either $1$-dimensional (if $L$ is unimodular) or $0$-dimensional (if $L$ is not unimodular). Moreover, if $\kappa = 2, 4, 6, 8, 10, 14$, i.e. $l = 6, 10, 14, 18, 22, 30$, then we have $\M_\kappa(\Gamma(L)) = \M_\kappa^{\partial \Eis}(\Gamma(L))$. In the other cases there could be additional holomorphic modular forms that are cusp forms on the boundary.
\end{cor}

\begin{proof}
Let $L$ be maximal of signature $(2, l), l \geq 6$. Then $L$ splits two hyperbolic planes over $\IZ$ and hence the map $\pi_L$ is bijective. Thus the space of holomorphic modular forms of singular weight which are linear combinations of Eisenstein series on the boundary has the same dimension as the space of invariant vectors in $\IC[L' / L]$. The latter is either $0$-dimensional (if $L$ is not unimodular) or $1$-dimensional (if $L$ is unimodular). Since we are in singular weight, there are no cusp forms. Since the $1$-dimensional boundary components of $\Gamma(L) \bs \IH_l$ are given by $\SL_2(\IZ) \bs \IH$ and since there are no cusp forms of weight $2, 4, 6, 8, 10, 14$ for $\SL_2(\IZ)$, we obtain $\M_\kappa(\Gamma(L)) = \M_\kappa^{\partial \Eis}(\Gamma(L))$ if $l = 6, 10, 14, 18, 22, 30$.
\end{proof}

\renewcommand\bibname{References}
\bibliographystyle{alphadin}
\bibliography{OrthogonalEisensteinSeriesAtHarmonicPointsAndModularFormsOfSingularWeight}

\end{document}